%% file: main.tex
\documentclass[11pt,a4paper]{article}
\usepackage[margin=2cm]{geometry}

\usepackage[utf8]{inputenc}
\usepackage[T1]{fontenc}
\usepackage{algorithm,algorithmic}
\usepackage{latexsym}
\usepackage{amsmath}
\usepackage{amsthm}
\usepackage{amssymb}
\usepackage{epsfig}
\usepackage{palatino}
\usepackage{bm}
\usepackage{microtype}
\usepackage{booktabs}
\usepackage{color}
\usepackage{enumitem}
\usepackage{multirow}
\usepackage[stable]{footmisc}
\usepackage{hyperref}
\usepackage[many]{tcolorbox}
\usepackage{pgfplots}
\pgfplotsset{compat=newest}
\usepackage{tikz,pgf}
\usetikzlibrary{backgrounds}
\usepackage{tikz-layers}
\usepackage{pstricks}
\usepackage{setspace} 

\usepackage[caption=false]{subfig}

\newtheorem{theorem}{Theorem}

\newtheorem{proposition}[theorem]{Proposition}

\theoremstyle{definition}
\newtheorem{definition}{Definition}

\numberwithin{equation}{section}

\def \Z{\mathbb{Z}}

\def \R{\mathbb{R}}

\def \N{\mathbb{N}}

\def \0{\bm{0}}

\def \d{\bm{d}}

\def \r{\bm{r}}
\def \s{\bm{s}}

\def \u{\bm{u}}

\def \x{\bm{x}}
\def \y{\bm{y}}
\def \z{\bm{z}}

\def \lmd{\bm{\lambda}}

\def \calE{\mathcal{E}}
\def \calF{\mathcal{F}}
\def \calG{\mathcal{G}}
\def \calI{\mathcal{I}}
\def \calK{\mathcal{K}}

\def \proj{\mathrm{proj}}

\DeclareMathOperator*{\argmin}{arg\,min}

\allowdisplaybreaks

\usepackage{endnotes}
\let\footnote=\endnote

\usepackage{natbib}
 \bibpunct[, ]{(}{)}{,}{a}{}{,}%

\newtheoremstyle{repeatthm}{}{}{\itshape}{}{\bfseries}{.}{.5em}{\thmnote{#3}}
\theoremstyle{repeatthm}
\newtheorem*{repeattheorem}{Theorem}

\title{An ADMM-based Distributed Optimization Method for Solving Security-Constrained AC Optimal Power Flow}

\author{Amin Gholami$^\textbf{a}$, Kaizhao Sun$^\textbf{b}$, Shixuan Zhang$^\textbf{c}$, Xu Andy Sun$^\textbf{d}$}
\date{
    $\textbf{a}$ H. Milton Stewart School of Industrial and Systems Engineering, Georgia Institute of Technology, Atlanta, Georgia 30332;
    $\textbf{b}$ DAMO Academy, Alibaba Group (U.S.), Inc., Bellevue, Washington 98004;
    $\textbf{c}$ Institute for Computational and Experimental Research in Mathematics, Brown University, Providence, Rhode Island 02903;
    $\textbf{d}$ Sloan School of Management, Massachusetts Institute of Technology, Cambridge, Massachusetts 02139}

\begin{document}

\maketitle

\begin{abstract}
In this paper, we study efficient and robust computational methods for solving the security-constrained alternating current optimal power flow (SC-ACOPF) problem, a two-stage nonlinear optimization problem with disjunctive constraints, that is central to the operation of electric power grids. The first-stage problem in SC-ACOPF determines the operation of the power grid in normal condition, while the second-stage problem responds to various contingencies of losing generators, transmission lines, and transformers. The two stages are coupled through disjunctive constraints, which model generators' active and reactive power output changes responding to system-wide active power imbalance and voltage deviations after contingencies. Real-world SC-ACOPF problems may involve power grids with more than $30$k buses and $22$k contingencies and need to be solved within 10-45 minutes to get a base case solution with high feasibility and reasonably good generation cost. 
We develop a comprehensive algorithmic framework to solve SC-ACOPF that meets the challenge of speed, solution quality, and computation robustness. In particular, we develop a smoothing technique to approximate disjunctive constraints into a smooth structure which can be handled by interior-point solvers; 
we design a distributed optimization algorithm to efficiently generate first-stage solutions;
we propose a screening procedure to prioritize contingencies; and finally, we develop a reliable and parallel architecture that integrates all algorithmic components. Extensive tests on industry-scale systems demonstrate the superior performance of the proposed algorithms.

\end{abstract}

\input{1Introduction}

\input{2Model}

\input{3AlgorithmStructure}

\input{4Smoothing}

\input{5DistributedOptimization}

\input{6ContingencyScreening}

\input{7OtherImplementations}

\input{8Computation}

\section{Conclusion}\label{sec:conclusion}
In this paper, we propose a novel and comprehensive algorithmic framework for solving large-scale SC-ACOPF problems under time and resource constraints. We propose smoothing techniques for handling disjunctive constraints that model the coupling of active and reactive power and voltages in the base case and post-contingency states. 
We introduce a two-level ADMM algorithm for solving the smoothed two-stage formulation of SC-ACOPF, whose iteration complexity can be reasonably quantified under certain technical conditions.
Another important component is the contingency screening algorithm that effectively predicts the cost of contingencies and provides the two-level ADMM with a subset of highly risky contingencies. Parallel computation is fully exploited. Extensive testing on the ARPA-E test cases demonstrates that the proposed algorithms can produce high quality solutions in a very restricted time framework for real-world sized power grids up to 30k buses and 22k contingencies.

\emph{Acknowledgment.}
The authors acknowledge the continued support of ARPA-E under award number DE-AR0001089. We would also like to sincerely thank Dr. Santanu S. Dey for the many inspiring discussions during the ARPA-E GO competition that greatly enriched our understanding of the challenging nature of the SC-ACOPF problem.

\bibliographystyle{plainnat} 
\bibliography{ref.bib} 

\clearpage
\appendix

\input{9Appendix}

\end{document}

%% file: 1Introduction.tex
\section{Introduction}
In this paper, we study efficient and robust computational methods for solving security constrained alternating current optimal power flow (SC-ACOPF) problems. 
The ability to solve such a problem in real-world large-scale power systems in a reliable and effective manner will bring various benefits, including reducing short-term operational costs, mitigating transmission congestion, increasing power system security, and improving power system operator's capability to conduct realistic long-term planning study.
Due to its central importance in grid optimization, ACOPF has a vast literature dating back to the 1960s. In the following, we briefly review the literature with a focus on recent advances in computational methods for solving ACOPF and then we summarize the present paper's contributions to the literature.

%
%

\subsection{Literature Review}

ACOPF was first formulated as an optimization problem in \cite{carpentier1962contribution}.
In a basic ACOPF problem, one seeks to 
determine the active and reactive power supply from the generators and the magnitude and phase angle of voltages in the power network in order to meet a given electricity load with the minimum generation cost, where the relation between electric power and voltages is described by alternating current (AC) power flow equations. Due to the highly nonlinear and nonconvex nature of the AC power flow equations, it is known that the class of ACOPF problems is NP-hard, even if the underlying power network has a tree structure~\citep{bienstock2019strong,lehmann2015ac}.

To reduce computational difficulty, a popular approach is to solve a relaxation or an approximation of the ACOPF~(see, e.g., \cite{molzahn2019survey,low2014convex1,low2014convex2} and the references therein).
This includes the direct current optimal power flow (DCOPF) model~\citep{stott2009dc}, linear programming approximation~\citep{coffrin2014linear}, second order conic programming relaxation~\citep{jabr2006radial,kocuk2016strong}, and semidefinite programming relaxation~\citep{bai2008semidefinite,lavaei2011zero}.
However, these approximations or relaxations, if used alone, could not guarantee to produce a feasible solution of the AC power flow equations, which is crucial for the practical applications of ACOPF.
To obtain a feasible solution to the ACOPF problem, it is common to use the interior point method (IPM)~\citep{wu1994direct,torres1998interior,jabr2002primal,wang2007computational}.
A benchmark IPM solver for ACOPF is MATPOWER by~\cite{MATPOWER}.

Security-constrained ACOPF (SC-ACOPF) problems consider post-contingency corrective actions that is coupled to the pre-contingency base case ACOPF \citep{capitanescu2011state}. It can be formulated as a two-stage nonlinear optimization problem, where the first-stage problem is the base case ACOPF problem for the normal operating condition, while the second-stage problem deals with a large set of generation and transmission contingencies.
While SC-ACOPF provides the security required in power system operation~\citep{capitanescu2016critical}, it also brings further challenges with significantly increased problems sizes and complicating constraints.
Common solution strategies for SC-ACOPF include post-contingency network compression~\citep{platbrood2013generic} and decomposition methods that exploit the contingency structure of SC-ACOPF~\citep{phan2013some}.

Moreover, various techniques have been developed to detect and rank a set of more critical contingencies on which the algorithm should focus 
\cite{2016-Baldick-stochastic, 2013-Capitanescu-contingency}.
For instance, in \cite{2016-Baldick-stochastic}, in addition to detecting important contingencies using a scenario identification index, similar contingencies are eliminated from contingency analysis.
We refer to \cite{2016-Baldick-stochastic, 2013-Capitanescu-contingency, 2007-capitanescu-contingency-filtering, 2013-fliscounakis-contingency} for a comprehensive survey on contingency screening.

Decomposition methods that are based on the augmented Lagrangian method (ALM) and its close variant, the alternating direction method of multipliers (ADMM), have shown favorable performance in solving ACOPF problems.
For example, early works \citep{kim1997coarse,baldick1999fast,kim2000comparison} applied an linearized ALM to a regional decomposition of ACOPF. 
\cite{peng2014distributed, peng2015distributed, peng2016distributed} applied ADMM to some convex relaxation of ACOPF on radial networks. 
The numerical success of ADMM has also been observed on nonconvex ACOPF \citep{chung2005implementing,chung2011multi,sun2013fully,erseghe2014distributed,mhanna_adaptive_2019} with the convergence studied under certain technical assumptions \citep{erseghe2014distributed, sun2021two}.
Despite recent studies on nonconvex ADMM \citep{wang2015global,hong2016convergence,jiang2014alternating}, convergence is established only under structural assumptions on the problem data, which cannot be satisfied by SC-ACOPF instances due to complicated constraints in both stages. This difficulty motivates us to go beyond the standard ADMM and propose a two-level framework to facilitate convergence. 

\subsection{Contributions}
This paper studies a class of
SC-ACOPF problems that have not been extensively studied in the literature. In particular, the SC-ACOPF model in this paper uses disjunctive constraints to model generators' active and reactive power output response to system-wide active power imbalance and voltage deviations due to contingencies. Such an SC-ACOPF model is more realistic than the existing ones in describing generators' post-contingency response, however, disjunctive constraints bring significant computational challenges that have not been encountered in the SC-ACOPF literature. 
In order to solve such an SC-ACOPF model, we develop a suite of innovative algorithmic techniques in a robust parallelized computation framework to achieve the combined goal of fast speed, high solution quality, and scalable computation. The contribution of the current paper can be summarized as below.


%
%
\begin{itemize}
    \item We have developed a smoothing technique to replace active and reactive power disjunctive constraints by approximate constraints defined by smooth functions so that efficient interior-point solvers can be utilized on the SC-ACOPF subproblems. 
    
    \item We design a distributed optimization algorithm to efficiently generate first-stage solutions under time limits, where contingencies' information are incorporated into first-stage decision making through updates of dual multipliers. We provide iteration complexity estimates to find an approximate stationary solution under some technical conditions.
     
    \item We propose a screening procedure to prioritize contingencies in order to handle extremely large systems within computational time limits, which has been tested on networks with up to 30k buses and 22k contingencies.
    We define a severity criterion to capture the infeasibilities of an operating point and the incurred penalties under each contingency. Our procedure is inspired by the fact that in practice for large-scale systems a base-case ACOPF solution might be treated as preventive (contrary to corrective measures), i.e., post-contingency rescheduling is not possible. We identify the contingencies for which a base-case ACOPF solution cause severe infeasibilities, hence the algorithm can focus on them to find a post-contingency solution.  
    
    \item We develop a parallel computation framework with various safeguarding mechanisms to ensure robust performance of the proposed algorithms. 
    
    \item We conduct extensive tests on industry-scale systems to demonstrate the superior performance of the proposed algorithms.
\end{itemize}
The paper is organized as follows. In Section \ref{sec:SC-ACOPF_Model}, we introduce the detailed SC-ACOPF model. Section \ref{sec:overallalg} gives an overview of the proposed algorithmic framework. Section \ref{sec:smoothing} introduces the smoothing technique for disjunctive constraints. Section \ref{sec:two_level_admm} proposes the two-level ADMM algorithm for SC-ACOPF with convergence analysis. Section \ref{sec:ranking} develops the contingency ranking algorithm. Section \ref{sec:Other Algorithmic Development} discusses the parallel computation architecture. Section \ref{sec:comp} presents extensive computational results. Section \ref{sec:conclusion} concludes the paper.

%% file: 2Model.tex
\section{The SC-ACOPF Model with Disjunctive Constraints}
\label{sec:SC-ACOPF_Model}

A power network is defined by a tuple $(\calI, \calE, \calF, \calG)$, where $\calI$ is the set of buses (nodes) in the power network, $\calE$ is the set of transmission lines, $\calF$ is the set of transformers, and $\calG$ is the set of generators connected to the buses. 
A \textit{contingency} is defined as an ACOPF subproblem with the underlying power network modified from the original one by removing a generator $g\in\calG$, a transmission line $e\in\calE$, or a transformer $f\in\calF$.
The set of buses remains unchanged for all contingencies. 
The ACOPF subproblem on the original power network is called the \textit{base case} (i.e. pre-contingency) indexed by $k=0$, and the contingencies are indexed by the set $\calK=\{1,2,\dots,|\calK|\}$. 
We call each member $k\in \bar{\calK}:=\calK\cup\{0\}$ a \textit{state} of the power system. Let $\calE_k$ ($\calF_k$, resp.) be the set of transmission lines (transformers, resp.) in state $k\in\bar{\calK}$. Naturally, in base case $\calE_0=\calE$, $\calF_0=\calF$. Let $G_k^P$ be the set of generators that can adjust its active power output in contingency $k\in\calK$. Let $G_{ik}\subset \calG$ be the set of generators connected to bus $i\in\calI$ in state $k\in\bar{\calK}$.

The decision variables include the active power output $p_{gk}$ and the reactive power output $q_{gk}$ of generator $g\in\calG$ in state $k\in \bar{\calK}$, the voltage magnitude $v_{ik}$ and phase angle $\theta_{ik}$, the active power flowing into an edge $h\in\calE_k\cup\calF_k$ at its origin (destination, resp.) end $p_{hk}^o$ ($p_{hk}^d$, resp.) in state $k\in \bar{\calK}$, and similarly the reactive power flows $q_{hk}^o$ and $q_{hk}^d$.
For each contingency $k\in \calK$, a scalar variable $\Delta_k$ models the system-wide active power imbalance caused by the contingency. Its role in determining the response of generators' active power output to the contingency will be described in details soon. The decision also includes slack variables in nodal active and reactive power balance $\sigma_{ik}^{P+}, \sigma_{ik}^{P-}, \sigma_{ik}^{Q+}, \sigma_{ik}^{Q-}$ and the slacks in transmission line and transformer current magnitude limit $\sigma_{hk}^S$ for $i\in \calI, h\in\calE_k\cup\calF_k, k\in\bar{\calK}$.

With the above notations, we can introduce the SC-ACOPF model that we study in this paper \cite{GOChallenge1_formulation}. We first introduce a compact form and then expand into some details.
\begin{subequations}
\begin{align}
    \min\quad & \sum_{g\in\calG} c_g(p_{g0}) + c_0^\sigma +  \frac{1}{|\calK|}\sum_{k\in\calK}c_k^\sigma \label{eq:scopf_obj} \\
    \text{s.t.}\quad & c_k^\sigma = \sum_{i\in\calI} \left(c_{ik}^p(\sigma_{ik}^{P+}+\sigma_{ik}^{P-})+c_{ik}^q(\sigma_{ik}^{Q+}+\sigma_{ik}^{Q-})\right) +  \sum_{e\in\calE}c_{ek}^S(\sigma_{ek}^S)+\sum_{f\in\calF}c_{fk}^S(\sigma_{fk}^S),\;\;\forall k\in\bar{\calK}, \label{eq:scopf_cost}\\
    & p_{ek}^o = p_{e}(v_{i_{e}^ok}, v_{i_{e}^dk}, \theta_{i_{e}^ok}, \theta_{i_{e}^dk}), \;\;
     p_{ek}^d = p_{e}(v_{i_{e}^dk}, v_{i_{e}^ok}, \theta_{i_{e}^dk}, \theta_{i_{e}^ok}), \;\;\forall e\in\calE_k, k\in\bar{\calK}, \label{eq:scopf_realpower_destination} \\
    & q_{ek}^o = q_{e}(v_{i_{e}^ok}, v_{i_{e}^dk}, \theta_{i_{e}^ok}, \theta_{i_{e}^dk}), \;\; q_{ek}^d = q_{e}(v_{i_{e}^dk}, v_{i_{e}^ok}, \theta_{i_{e}^dk}, \theta_{i_{e}^ok}), \;\;\forall e\in\calE_k, k\in\bar{\calK}, \label{eq:scopf_reactivepower_destination} \\
    & p_{fk}^o = p_{f}^o(v_{i_{f}^ok}, v_{i_{f}^dk}, \theta_{i_{f}^ok}, \theta_{i_{f}^dk}), \;\; p_{fk}^d = p_{f}^d(v_{i_{f}^dk}, v_{i_{f}^ok}, \theta_{i_{f}^dk}, \theta_{i_{f}^ok}), \;\;\forall f\in\calF_k, k\in\bar{\calK}, \label{eq:scopf_realpower_tx_destination} \\
    & q_{fk}^o = q_{f}^o(v_{i_{f}^ok}, v_{i_{f}^dk}, \theta_{i_{f}^ok}, \theta_{i_{f}^dk}), \;\; q_{fk}^d = q_{f}^d(v_{i_{f}^dk}, v_{i_{f}^ok}, \theta_{i_{f}^dk}, \theta_{i_{f}^ok}), \;\;\forall f\in\calF_k, k\in\bar{\calK}, \label{eq:scopf_reactivepower_tx_destination} \\
    & \sum_{g\in G_{ik}}p_{gk} - p_{ik}^L = \sum_{e\in E_{ik}^o}p_{ek}^o + \sum_{e\in E_{ik}^d}p_{ek}^d + \sum_{f\in F_{ik}^o}p_{fk}^o + \sum_{f\in F_{ik}^d}p_{fk}^d + \sigma_{ik}^{P+} - \sigma_{ik}^{P-},\;\;\forall i\in\calI, k\in\bar{\calK}, \label{eq:scopf_realpower_nodal} \\
    & \sum_{g\in G_{ik}}q_{gk} - q_{ik}^L = \sum_{e\in E_{ik}^o}q_{ek}^o + \sum_{e\in E_{ik}^d}q_{ek}^d + \sum_{f\in F_{ik}^o}q_{fk}^o + \sum_{f\in F_{ik}^d}q_{fk}^d + \sigma_{ik}^{Q+} - \sigma_{ik}^{Q-},\;\;\forall i\in\calI, k\in\bar{\calK}, \label{eq:scopf_reactivepower_nodal} \\
    & \sqrt{(p_{ek}^o)^2 + (q_{ek}^o)^2}\le \bar{R}_{ek} v_{i_{e}^ok} + \sigma_{ek}^S, \;\;\forall e\in\calE_k, k\in\bar{\calK}, \label{eq:scopf_linepowerlimit_origin}\\
    & \sqrt{(p_{ek}^d)^2 + (q_{ek}^d)^2}\le \bar{R}_{ek} v_{i_{e}^dk} + \sigma_{ek}^S, \;\;\forall e\in\calE_k, k\in\bar{\calK}, \label{eq:scopf_linepowerlimit_destination}\\
    & \sqrt{(p_{fk}^o)^2 + (q_{fk}^o)^2}\le \bar{s}_{fk}  + \sigma_{fk}^S, \;\;\forall f\in\calF_k, k\in\bar{\calK}, \label{eq:scopf_txpowerlimit_origin}\\
    & \sqrt{(p_{fk}^d)^2 + (q_{fk}^d)^2}\le \bar{s}_{fk} + \sigma_{fk}^S, \;\;\forall f\in\calF_k, k\in\bar{\calK}, \label{eq:scopf_txpowerlimit_destination}\\
    & p_{gk} = \proj_{[\underline{p}_g,\bar{p}_g]}(p_{g0} + \alpha_g \Delta_k), \;\; \forall g\in G_k, k\in \calK, \label{eq:scopf_realpowercntg}\\
    & \left\{\underline{q}_g \le q_{gk} \le \bar{q}_g, \; v_{i_{g}k}=v_{i_{g}0}\right\}\cup\left\{q_{gk}=\bar{q}_k, \; v_{i_{g}k}\le v_{i_g0}\right\}\cup\left\{q_{gk}=\underline{q}_k, \; v_{i_{g}k}\ge v_{i_g0}\right\},\;\forall g\in G_k, k\in\calK,\label{eq:scopf_PVPQswitch}\\
    & \underline{v}_{ik}\le v_{ik}\le \bar{v}_{ik}, \;\; \underline{p}_g\le p_{gk} \le \bar{p}_{g}, \;\; \underline{q}_g\le q_{gk} \le \bar{q}_{g}, \;\;  \forall i\in \calI, g\in G_k, k\in\bar\calK, \label{eq:scopf_bounds_vpq}\\
    & p_{gk}=0, \;\; q_{gk}=0, \;\;\forall g\in\calG\setminus G_k, k\in\calK, \label{eq:scopf_inactive}\\
    & \sigma_{ik}^{P+}\ge 0, \;\; \sigma_{ik}^{P-}\ge 0, \;\;\sigma_{ik}^{Q+}\ge 0, \;\; \sigma_{ik}^{Q-}\ge 0, \;\; \sigma_{ek}^{S}\ge 0, \;\; \sigma_{fk}^{S}\ge 0, \;\; \forall i\in \calI, e\in\calE_k, f\in\calF_k, k\in\bar\calK. \label{eq:scopf_slackbounds}
\end{align}\label{eq:scopf}
\end{subequations}

Now let us explain the various functions and constraints in the above SC-ACOPF formulation. The objective function \eqref{eq:scopf_obj} is the sum of the total active power generation cost in the base case in the first term and the weighted sum of the base case penalty cost $c_0^\sigma$ in the second term and the average contingency cost in the third term. The generation cost $c_g(p_{g0})$ is a convex piecewise linear (pwl) increasing function of the active power generation $p_{g0}$ of generator $g$ in base case. The number of linear pieces of $c_g(\cdot)$ can vary with the generator $g$. 
The penalty cost $c_k^\sigma$ for a state $k\in\bar{\calK}$ is defined in constraint \eqref{eq:scopf_cost}, which is composed of three terms. The first term is the total penalty cost of all the active and reactive power slack variables $\sigma_{ik}^{P\pm}$ and $\sigma_{ik}^{Q\pm}$ used in bus power balance constraints \eqref{eq:scopf_realpower_nodal}-\eqref{eq:scopf_reactivepower_nodal}. In particular, the functions $c_{ik}^p(\cdot)$ and $c_{ik}^q(\cdot)$ are convex pwl increasing functions with three pieces, where the first piece has a small penalty price for minor power balance violations, the second piece has more stringent penalty for moderate violations, while the last piece has enormous penalty for remaining violation to encourage the solution to have no constraint violation. The second and the third terms in \eqref{eq:scopf_cost} have two similar convex pwl increasing functions on the slack variables $\sigma_{ek}^S$ appeared in line current magnitude limits \eqref{eq:scopf_linepowerlimit_origin}-\eqref{eq:scopf_linepowerlimit_destination} and $\sigma_{fk}^S$ appeared in transformer current magnitude limits \eqref{eq:scopf_txpowerlimit_origin}-\eqref{eq:scopf_txpowerlimit_destination}. 

Constraints \eqref{eq:scopf_realpower_destination}-\eqref{eq:scopf_reactivepower_destination} are the equations that calculate the active and reactive power flowing into a transmission line at one end of the line. For a line $e\in\calE_k$ in state $k\in\bar{\calK}$, the active power function $p_{e}(a,b,c,d)$ in \eqref{eq:scopf_realpower_destination} has the following form
\begin{subequations}
\begin{align}
    p_{e}(a, b, c, d) = g_e a^2 + (-g_e\cos(c-d)-b_e\sin(c-d))ab.\label{eq:scopf_pe}
\end{align}
The reactive power function $q_e(a,b,c,d)$ in \eqref{eq:scopf_reactivepower_destination} is given by
\begin{align}
    q_{e}(a, b, c, d) = -(b_e + b_e^{CH}/2) a^2 + (b_e\cos(c-d)-g_e\sin(c-d))ab.\label{eq:scopf_qe}
\end{align}
\end{subequations}
In \eqref{eq:scopf_pe}-\eqref{eq:scopf_qe}, the parameters $g_e$ and $b_e$ are the series conductance and susceptance of line $e$, respectively, and $b_e^{CH}$ is the charging susceptance of line $e$ in a $\pi$-model of a transmission line \cite[Section 3.2]{MATPOWER}. Notice that the power flowing into a line $e$ at its destination end in  \eqref{eq:scopf_realpower_destination} and \eqref{eq:scopf_reactivepower_destination} has the same function as the power flowing at the origin end of $e$. The only difference is the voltage magnitudes at the origin and destination ends are switched, so are the angles. 

Constraints \eqref{eq:scopf_realpower_tx_destination}-\eqref{eq:scopf_reactivepower_tx_destination} are the equations that calculate the active and reactive power flowing into a transformer at one end of the transformer. 
The forms of $p_f^o$ and $p_f^d$ (resp.\ $q_f^o$ and $q_f^d$) are similar to $p_e$ in \eqref{eq:scopf_pe} (resp.\ $q_e$ in \eqref{eq:scopf_qe}), but are slightly more complicated (e.g., without the symmetry as $p_{f}^o\neq p_{f}^d$ and $q_{f}^o\neq q_{f}^d$). 
We omit the details here, which can be found in \cite[Section 3.6.5]{GOChallenge1_formulation}.

Constraints \eqref{eq:scopf_realpower_nodal} and \eqref{eq:scopf_reactivepower_nodal} are the nodal active and reactive power balance, respectively. Note that the slack variables $\sigma_{ik}^{P\pm}$ and $\sigma_{ik}^{Q\pm}$ appear in these equations to allow violation of power flow balance. However, recall that the slacks are heavily penalized in the objective function. Therefore, the optimal solution tends to satisfy the nodal power balance without slack. 
In fact, a slightly more general model, where the nodal power balance equations \eqref{eq:scopf_realpower_nodal} and \eqref{eq:scopf_reactivepower_nodal} contain additional terms of shunts, is used in all of our computational experiments in Section~\ref{sec:comp}.
To avoid overly complicate the formulation, we refer any interested readers to the detailed documentation~\citep{GOChallenge1_formulation} for more information on the shunt modeling.

Constraints \eqref{eq:scopf_linepowerlimit_origin}-\eqref{eq:scopf_txpowerlimit_destination} impose the current magnitude limits on transmission lines and transformers. The parameter $\bar{R}_{ek}$ is the maximum current magnitude allowed on line $e$ in state $k$, while $\bar{s}_{fk}$ is the maximum current magnitude allowed through transformer $f$ in state $k$. The slack variables $\sigma_{ek}^S$ and $\sigma_{fk}^S$ are introduced to allow violation of these constraints. 
The bounds \eqref{eq:scopf_bounds_vpq} on the voltages, active power, and reactive power output and \eqref{eq:scopf_slackbounds} on the slack variables are standard. Constraint \eqref{eq:scopf_inactive} dictates that no active and reactive power is produced if a generator is not active in a contingency. 

The key feature that differentiates the model \eqref{eq:scopf} from other SC-ACOPF models is the two constraints \eqref{eq:scopf_realpowercntg} and \eqref{eq:scopf_PVPQswitch}, which describe how the generator active and reactive power respond to contingencies. Specifically, the first one \eqref{eq:scopf_realpowercntg} models the active power response of an active generator $g$ in contingency $k$. It says that the active power generation $p_{gk}$ is equal to $p_g + \alpha_g \Delta_k$, if  $p_{gk}$ is between the two bounds $[\underline{p}_g,\bar{p}_g]$, and is at one of the bounds, otherwise. A plot of $p_{gk}$ as a function of $\Delta_k$ is given in Figure \ref{fig:pgk} in Section \ref{sec:smoothing}. The intuition is that the active power balance established in the base case would be broken in a contingency by, say, the loss of a generator. This would create an active power imbalance, which is accounted for by $\Delta_k$ (which also includes the thermal loss of the network, i.e., the power loss due to the heat generated by the current). The parameter $\alpha_g$ is a pre-determined participation factor for generator $g$ to take up a portion of the total imbalance $\Delta_k$. Note that we do not restrict the sign of $\Delta_k$, which could be negative, e.g. in a transmission line contingency. Constraint \eqref{eq:scopf_realpowercntg} is a disjunctive constraint, and can be reformulated by mixed integer constraints. See \cite[Section 3.14.3]{GOChallenge1_formulation}. One of the key innovations of this study is to \textit{not} use mixed integer reformulations, but rather to exploit a \textit{smooth approximation} of the constraint \eqref{eq:scopf_realpowercntg}, which is amenable for interior-point solvers. The details are given in Section \ref{sec:smoothing}.  

The second constraint \eqref{eq:scopf_PVPQswitch} models the response of reactive power output and voltage magnitude of a generator in a contingency. In particular, a generator's reactive power output should be used to hold the bus voltage in the contingency at the same voltage level of the base case as much as possible, until the generator's reactive power output is at its upper or lower bound. This model is motivated by the behavior of a generator's local voltage controller (see Section 2.2.2 in~\cite{2022-intro-paper}).
Note that this is also a disjunctive constraint and can be reformulated as a set of mixed integer constraints. To avoid integer variables, we also introduce smooth approximations to simplify computation.

Overall, the SC-ACOPF model \eqref{eq:scopf} is often formulated as a two-stage mixed integer nonlinear program (MINLP). 
To the best of our knowledge, no off-the-shelf solver could achieve a highly feasible and near-optimal base case solution for large-scale problems in power grid optimization within a stringent time requirement of $10$ to $45$ minutes. 
Starting from the next section, we will describe a practical algorithmic framework and demonstrate that it achieves the above goals with robust performance in a large number of test cases of industry sizes up to 30k buses and 22k contingencies.

%% file: 3AlgorithmStructure.tex
\section{Overall Algorithmic Structure}\label{sec:overallalg}
In this section, we give an overview of the proposed algorithmic framework for solving the SC-ACOPF problem \eqref{eq:scopf}, which consists of a two-step strategy. In the first step, we aim to find a base case operating point under a strict time limit, where
the base case solution should also take into account of the effects of post-contingency corrective actions. To this end, we first propose a novel and implementable smooth approximation to reformulate the disjunctive constraints \eqref{eq:scopf_realpowercntg}-\eqref{eq:scopf_PVPQswitch} between base case and contingencies (Section \ref{sec:smoothing}). Then we combine the proposed smooth approximation with a two-level ADMM algorithm to find a base case solution that considers a subset of contingencies, selected from a contingency screening procedure (Section \ref{sec:two_level_admm}). The contingency screening algorithm (Section \ref{sec:ranking}) measures and ranks the severity of each contingency. The convergence of the two-level ADMM algorithm is established under suitable assumptions.

\begin{algorithm}[h!]
	\caption{: An Overview}\label{alg:overview}
	\begin{algorithmic}[1]
            \STATE Select a subset $\cal{K}'\subseteq\cal{K}$ using Algorithm \ref{alg:ctg_ranking} ({Contingency Ranking}) if $|\cal{K}|$ is large;
	    \STATE Obtain the base case solution by Algorithm \ref{alg:two_level} ({Two-level ADMM}) with contingencies  $\cal{K}'$;
	    \STATE Sort contingency list $\cal{K}$ by Algorithm \ref{alg:ctg_ranking} ({Contingency Ranking});
	    \FOR{$k$ in the sorted contingency list $\cal{K}$}\label{alg:overview:ContingencyForLoop}
	    \STATE Obtain the solution for contingency $k$ by Algorithm \ref{alg:recourse} ({Recourse Model Solution Strategy});
	    \ENDFOR
	\end{algorithmic}
\end{algorithm}
In the second step, given the base case operating point reported in the first step, we aim to recover power flow solutions for all contingencies. 
Contingencies are ranked by the contingency screening algorithm given the base case solution found in the first step. Contingencies with high rankings are more likely to give large penalties and will be solved before those with low rankings. Details will be discussed in Section \ref{sec:ranking}.

%% file: 4Smoothing.tex
\section{Smooth Approximation of Disjunctive Constraints}\label{sec:smoothing}

In problem~\eqref{eq:scopf}, all constraints but~\eqref{eq:scopf_realpowercntg} and \eqref{eq:scopf_PVPQswitch} are defined by twice continuously differentiable functions.
In order to use interior-point method solvers, we propose a smooth approximation of the constraints~\eqref{eq:scopf_realpowercntg} and~\eqref{eq:scopf_PVPQswitch} in this section.

To begin with, consider the univariate function \(F(x):=\max\{0,x\}\) for \(x\in\R\), which is not differentiable at \(x=0\).
Picking any positive number \(\epsilon>0\), we can approximate \(F(x)\) by the smooth function \(F^\epsilon(x):=\epsilon\ln(1+\exp(x/\epsilon))\) for \(x\in\R\), with the easily verifiable bound on the approximation error:
\begin{equation}\label{eq:SmoothedApproximation}
    F^\epsilon(x)-\epsilon\ln2\le F(x)\le F^\epsilon(x),\quad\forall\,x\in\R.
\end{equation}
Thus the approximation quality is uniformly controlled by the chosen parameter \(\epsilon>0\).
This approximation turns out to have a long history in nonlinear complementarity problems \citep{chen1996class} and variational analysis \citep{rockafellar2009variational}, with numerous applications in machine learning \citep{lee2001ssvm,schmidt2007fast,chen2012smoothing}.
Now recall that for each \(g\in G_k\) and \(k\in\calK\), the active power response of an active generator is modeled as
\begin{equation}
    p_{gk}(\Delta_k)
    =\proj_{[\underline{p}_g,\bar{p}_g]}(p_{g0} + \alpha_g \Delta_k)
    =\max\{\underline{p}_g,-\max\{-\bar{p}_g,-p_{g0}-\alpha_g\Delta_k\}\},
\end{equation}
which is a composition of maximum functions.
Therefore, we apply the smooth approximation discussed above and obtain the approximate active power response function as
\begin{equation}\label{eq:SmoothedRealPowerDisjunction}
    {p}_{gk}^\epsilon(\Delta_k):=\underline{p}_g+\epsilon\ln\left[1+\frac{\exp[(\bar{p}_g-\underline{p}_g)/\epsilon]}{1+\exp[(\bar{p}_g-p_{g0}-\alpha_g\Delta_k)/\epsilon]}\right].
\end{equation}
Moreover, if we assume that \(p_{g0}+\alpha_g\Delta_k\ge\underline{p}_g\) for all \(g\in G_k\), then the active power response function \(p_{gk}(\Delta_k)\) can be simplified as \(\tilde{p}_{gk}(\Delta_k)=\min\{\bar{p}_g,p_{g0}+\alpha_g\Delta_k\}\).
The approximation in this case becomes
\begin{equation}\label{eq:real_power_smooth_approx_lower_bounded}
    \tilde{p}_{gk}^{\epsilon}(\Delta_k):=\bar{p}_g-\epsilon\ln\left[1+\exp\left(\frac{\bar{p}_g-p_{g0}-\alpha_g\Delta_k}{\epsilon}\right)\right].
\end{equation}
We illustrate the true and approximate active power response functions for this case in Figure~\ref{fig:pgk}.
\begin{figure}[htbp]
    \centering
    \includegraphics[scale=1]{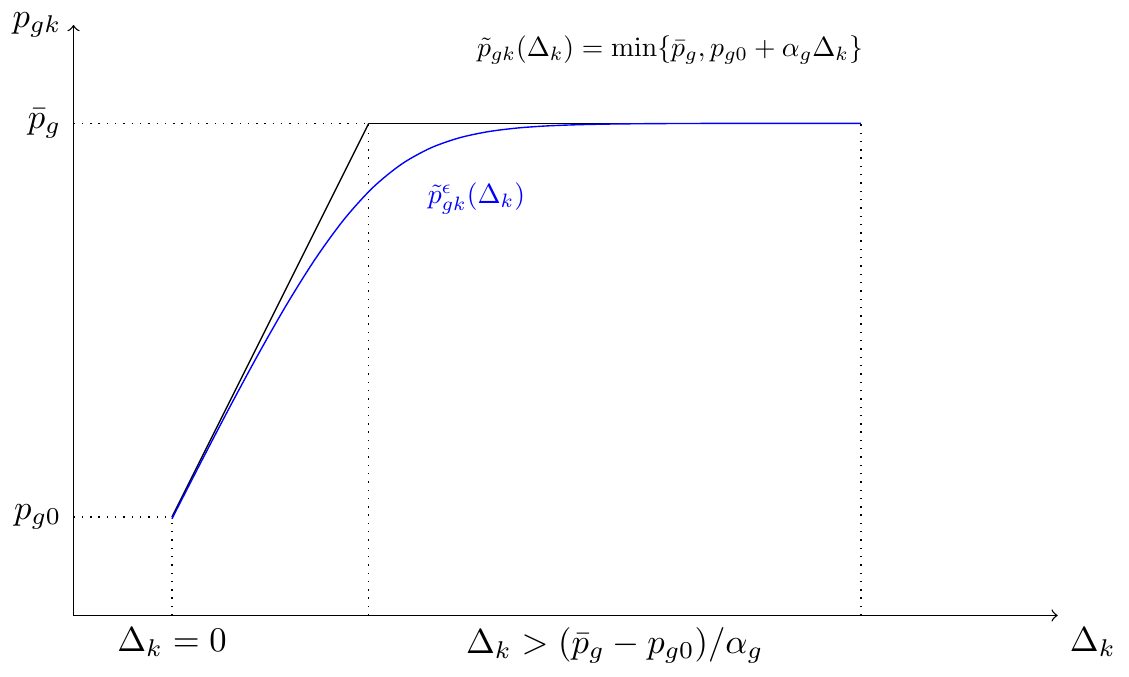}
    \caption{True and Approximate Active Power Response Functions Assuming \(p_{g0}+\alpha_g\Delta_k\ge\underline{p}_g\) for Generator \(g\in G_k\).}
    \label{fig:pgk}
\end{figure}
The following proposition shows that \({p}_{gk}^\epsilon(\Delta_k)\) (resp.\,\(\tilde{p}_{gk}^{\epsilon}(\Delta_k)\)) is a consistent approximation of the function \(p_{gk}(\Delta_k)\) (resp.\,\(\tilde{p}_{gk}(\Delta_k)\)).
\begin{proposition}\label{prop:SmoothedRealDisjunction}
    For each \(g\in G_k\) and \(k\in\calK\), as $\epsilon\to 0$, the approximation errors converge to zero, i.e.,
    \begin{align*}
        \sup\left\{\vert{p}_{gk}^\epsilon(\Delta_k)- p_{gk}(\Delta_k)\vert:\Delta_k\in\R\right\}\to 0, \;\; \sup\left\{\vert\tilde{p}_{gk}^{\epsilon}(\Delta_k)- \tilde{p}_{gk}(\Delta_k)\vert:\Delta_k\ge (\underline{p}_g-p_{g0})/\alpha_g\right\}\to 0.
    \end{align*}
\end{proposition}
The proof is presented in Section~\ref{sec:Proof:SmoothedRealDisjunction}.

The reactive power response constraint~\eqref{eq:scopf_PVPQswitch} can be similarly approximated and relaxed.
Define the feasibility set for the reactive power response as \(S_{gk}:=\{\underline{q}_g \le q_{gk} \le \bar{q}_g, \; v_{i_{g}k}=v_{i_{g}0}\}\cup\{q_{gk}=\bar{q}_k, \; \underline{v}_{i_g}\le v_{i_{g}k}\le v_{i_g0}\}\cup\{q_{gk}=\underline{q}_k, \; v_{i_g 0}\le v_{i_{g}k}\le \bar{v}_{i_g}\}\) for each \(g\in G_k\) and \(k\in\calK\).
Note that \(S_{gk}\) can be equivalently reformulated as the following constraints:
\begin{subequations}
\begin{align}
    (q,v)\in S_{gk}\iff
    \exists\,0\le v_{i_g}^+,v_{i_g}^-\le \bar{v}_{i_g}-\underline{v}_{i_g},\;\mathrm{s.t.}\quad
    &v_{i_gk}=v_{i_g0}+v_{i_g}^+-v_{i_g}^-,\label{eq:SmoothedApproximation:SetConstr1}\\
    &\min\{q_{gk}-\underline{q}_g,v_{i_g}^+\}\le0,\label{eq:SmoothedApproximation:SetConstr2}\\
    &\min\{-q_{gk}+\bar{q}_g,v_{i_g}^-\}\le0,\label{eq:SmoothedApproximation:SetConstr3}\\
    &\underline{q}_g\le q_{gk}\le \bar{q}_g,\; \underline{v}_{i_g}\le v_{i_g k}\le \bar{v}_{i_g}.\label{eq:SmoothedApproximation:SetConstr4}
\end{align}
\end{subequations}
Given any \(\epsilon >0\), similar to the active power response approximation, we may define a relaxation set \({S}^{\epsilon}_{gk}\), consisting of the following continuously differentiable constraints
\begin{subequations}\label{eq:SmoothedConstraintApproximation}
\begin{align}
    (q,v)\in {S}^{\epsilon}_{gk}\iff
    \exists\,0\le v_{i_g}^+,v_{i_g}^-\le \bar{v}_{i_g}-\underline{v}_{i_g},\;\mathrm{s.t.}\quad
    &v_{i_gk}=v_{i_g0}+v_{i_g}^+-v_{i_g}^-,\label{eq:SmoothedApproximation:ApproxConstr1}\\
    &v_{i_g}^+-\epsilon\ln\left[1+\exp\left(\frac{v_{i_g}^+-q_{gk}+\underline{q}_{g}}{\epsilon}\right)\right]\le\epsilon\ln2,\label{eq:SmoothedApproximation:ApproxConstr2}\\
    &v_{i_g}^--\epsilon\ln\left[1+\exp\left(\frac{v_{i_g}^-+q_{gk}-\bar{q}_{g}}{\epsilon}\right)\right]\le\epsilon\ln2,\label{eq:SmoothedApproximation:ApproxConstr3}\\
    &\underline{q}_g\le q_{gk}\le \bar{q}_g,\; \underline{v}_{i_g}\le v_{i_g k}\le \bar{v}_{i_g}.\label{eq:SmoothedApproximation:ApproxConstr4}
\end{align}
\end{subequations}

\begin{proposition}\label{prop:SmoothedReactiveDisjunction}
    For each \(g\in G_k\) and \(k\in\calK\), given any \(\epsilon>0\), we have \(S_{gk}\subseteq {S}_{gk}^{\epsilon}\).
    Moreover, the distance \(\sup_{(q,v)\in {S}_{gk}^{\epsilon}}\inf_{(q',v')\in S_{gk}}\Vert (q,v)-(q',v')\Vert\to 0\) as \(\epsilon\to 0\).
\end{proposition}
The proof is presented in Section~\ref{sec:Proof:SmoothedReactiveDisjunction}.

While it appears tempting from these propositions to choose a small \(\epsilon\) for higher approximation accuracy, we point out by the following calculation the potential numerical issues.
In the general form, the approximation \(F^\epsilon(x)\) of the function \(F(x)=\max\{0,x\}\) has a second-order derivative 
\[
\frac{\mathrm{d}^2}{\mathrm{d} x^2} F^\epsilon(0)=\left(\frac{1}{\epsilon}\frac{\exp(x/\epsilon)}{(1+\exp(x/\epsilon))^2}\right)\bigg\vert_{x=0}=\frac{1}{4\epsilon}\to\infty,\quad\text{ as }\epsilon\to 0.
\]
Therefore, we face the trade-off between approximation accuracy and numerical condition.
We present our choice of \(\epsilon\) in the numerical experiments in Section~\ref{sec:Weights_and_Penalty Function_Parameters}.

%% file: 5DistributedOptimization.tex
\section{Distributed Optimization for Solving SC-ACOPF}\label{sec:two_level_admm}

\subsection{Distributed Reformulation}
In this section, we propose an ADMM-based distributed algorithm for finding an approximate stationary solution to a smoothed version of the SC-ACOPF problem \eqref{eq:scopf}.
We first introduce some notations to simplify the presentation. Let
\begin{align}
	\x_k := \Big(\{v_{ik}, \theta_{ik}, \sigma^{P\pm}_{ik}, \sigma^{Q\pm}_{ik}\}_{i \in \mathcal{I}}, ~\{p_{gk}, q_{gk}\}_{g \in G_k}, ~ \{\sigma^S_{ek}\}_{e\in \mathcal{E}_k}, \{\sigma^S_{fk}\}_{ f\in \mathcal{F}_k}, \Delta_k\Big)
\end{align}
be a column vector that consists of all variables in state $k\in \bar{\mathcal{K}}$, where we set $\Delta_0 = 0 $ for notational consistency. The OPF constraints in the base case can be compactly expressed as 
\begin{align}\label{eq:base_case_constr}
	X_0 := \{\x_0: \eqref{eq:scopf_realpower_destination}-\eqref{eq:scopf_txpowerlimit_destination}, \eqref{eq:scopf_bounds_vpq},\eqref{eq:scopf_slackbounds}\},
\end{align}
where, allowing a minor abuse of notation, constraints in $X_0$ are meant to be satisfied for the base case only. Notice that constraints \eqref{eq:scopf_realpowercntg} and \eqref{eq:scopf_PVPQswitch} involve both contingency variables $\x_k$ and base case variables $\{p_{g0},v_{i_g 0}\}_{g \in G_k}$; in order to formulate the couplings, let $\tilde{X}_k := \{(\x_k, \x_0): \eqref{eq:scopf_realpower_destination}-\eqref{eq:scopf_slackbounds}\}$
denote the feasible region of contingency variables $\x_k$ for each $k\in \mathcal{K}$, which also depends on the base case variables $\x_0$. The proposed algorithm requires some nonlinear optimization solver as the subproblem solution oracle. However, as we mentioned earlier, nonlinear solvers are not able to directly handle nonsmooth and disjunctive constraints in the form of \eqref{eq:scopf_realpowercntg} and \eqref{eq:scopf_PVPQswitch}, and hence $\tilde{X}_k$ defined earlier is not implementable. Therefore, we propose to replace $\tilde{X}_k$ by some proper workaround. For constraint \eqref{eq:scopf_PVPQswitch}, we restrict the contingency voltage variable $v_{i_g k}$ and reactive power $q_{gk}$ in the first disjunction, so that $q_{g k}$ is free to be dispatched within its range. For constraint \eqref{eq:scopf_realpowercntg}, we shall replace it by either the smooth approximation technique introduced in Section \ref{sec:smoothing}, or a simple continuous relaxation of the mixed-integer representation (often known as the ``big-M'' formulation, see Section 2.10 by~\citet{conforti2014integer} for more details). 
Denote the resulting approximated contingency feasible region by $X_k$ for $k\in \mathcal{K}$. 
Further define cost functions
\begin{align}\label{eq:abstrac_obj}
	f_0(\x_0):=  \sum_{g\in\mathcal{G}} c_g(p_{g0}) + \delta c_0^\sigma, \ \text{and} \ f_k(\x_k):= (1-\delta) \frac{1}{|\mathcal{K}|} c_k^\sigma,~\forall k\in K,
\end{align}
where $c_k^\sigma$ is given in \eqref{eq:scopf_cost}. Now we can abstract a smooth approximation of the SC-ACOPF problem \eqref{eq:scopf} as:
\begin{subequations}\label{eq:two_block_scopf}
\begin{align}
	\min \quad & f_0(\x_0) + \sum_{k \in \mathcal{K}} f_k(\x_k)	\\
	\mathrm{s.t.} \quad &  \x_0 = \x^{\text{base}}_k, \ \forall k \in \mathcal{K},\label{eq:two_blocl_couple} \\ 
	& \x_0 \in X_0,\ (\x_k,\x^{\text{base}}_k) \in X_k,\ \forall k \in \mathcal{K},
\end{align}
\end{subequations}
where $\x^{\text{base}}_k$ is a copy of the base variable $\x_0$ kept by contingency state $k \in \mathcal{K}$. 
Before we move on to the algorithmic development, we note that formulation \eqref{eq:two_block_scopf} is not equivalent to the original SC-ACOPF \eqref{eq:scopf} since disjunctive constraints are replaced by their smoothed or relaxed versions in $X_k$ for $k\in \mathcal{K}$. As we will elaborate more in Section \ref{subsec: Evaluation Methods}, the proposed algorithm is used to generate a first-stage solution $x_0$ that incorporates contingencies' information. For this purpose, the adoption of formulation \eqref{eq:two_block_scopf} is reasonable since 1) the base case solution generated in each iteration of the proposed algorithm remains feasible, i.e., $X_0$ is unmodified,  and 2) $X_k$'s preserve coupling natures in the original problem \eqref{eq:scopf} to some extend with an enhanced computational tractability. 

\subsection{An ADMM-based Decomposition Algorithm}
Problem \eqref{eq:two_block_scopf} has two blocks of variables, $\x_0$ and $\{\x_k, \x^{\text{base}}_k\}_{k\in \mathcal{K}}$, where variables $\{\x^{\text{base}}_k\}_{k\in \mathcal{K}}$ further admit a block-angular structure in constraint \eqref{eq:two_blocl_couple}. A popular method for decomposing large-scale problems in this form is the alternating direction method of multipliers (ADMM). ADMM minimizes the augmented Lagrangian function alternatingly with respect to $\x_0$ and $\{\x_k, \x^{\text{base}}_k\}_{k\in \mathcal{K}}$. Notice that when $\x_0$ is fixed, the updates of $\{\x_k, \x^{\text{base}}_k\}_{k\in \mathcal{K}}$ are decoupled among $k\in \mathcal{K}$. In this way computation can be distributed. However, it is known that ADMM directly applied to nonconvex problems like \eqref{eq:two_block_scopf} is not guaranteed to converge (see e.g. \cite{wang2015global}).

Now we propose a distributed algorithm with guaranteed global convergence to solve \eqref{eq:two_block_scopf}. The key idea consists of three steps. 
Firstly, we consider the following relaxation of \eqref{eq:two_block_scopf} with three blocks of variables $\x_0$, $\{\x_k, \x^{\text{base}}_k\}_{k\in \mathcal{K}}$, and a new slack variable $\{{\z}_k\}_{k\in \mathcal{K}}$:
\begin{subequations}\label{eq:three_block_relaxation}
\begin{align}
	\min_{\x_0, \{\x_k, \x^{\text{base}}_k\}_{k\in \mathcal{K}}, \{{\z}_k\}_{k\in \mathcal{K}}} \quad & f_0(\x_0) + \sum_{k \in \mathcal{K}} \left(f_k(\x_k)	+\langle \lmd_k, {\z}_k \rangle + \frac{\beta}{2}\|{\z}_k\|^2\right)\label{eq:three_block_relaxation_obj}\\
	\mathrm{s.t.} \quad &  \x_0 - \x^{\text{base}}_k+ {\z}_k = 0, \ \forall k \in \mathcal{K}, \label{eq:three_block_couple_constr}\\
	& \x_0 \in X_0,\ (\x_k,\x^{\text{base}}_k) \in X_k,\ \forall k \in \mathcal{K},
\end{align}
\end{subequations}
where ${\lmd}:=\{\lmd_k\}_{k\in \mathcal{K}}$ and $\beta >0$ are parameters. Notice that problem \eqref{eq:three_block_relaxation} is equivalent to problem \eqref{eq:two_block_scopf} if we explicitly enforce slack variables ${\z}_k$ to be zero as
\begin{align}\label{eq:slack_zero}
	{\z}_k = 0,\quad \forall k\in \mathcal{K}.
\end{align}
Therefore, \eqref{eq:three_block_relaxation} is actually the augmented Lagrangian relaxation (ALR) of \eqref{eq:two_block_scopf}, where constraints \eqref{eq:slack_zero} are relaxed with dual variables $\{\lmd_k\}_{k\in \mathcal{K}}$ and penalized with some $\beta >0$ as in \eqref{eq:three_block_relaxation_obj}.

Secondly, we apply a 3-block ADMM algorithm to solve the augmented Lagrangian relaxation problem \eqref{eq:three_block_relaxation}. Let $\{\y_k\}_{k\in \mathcal{K}}$ be the dual variables corresponding to the coupling constraints \eqref{eq:three_block_couple_constr}. Then given some $\rho >0$, the augmented Lagrangian function associated with problem \eqref{eq:three_block_relaxation} is defined as: 
\begin{align}\label{eq:al_function}
 L_\rho(\x_0, \{\x_k, \x^{\text{base}}_k\}_{k\in \mathcal{K}}, \{{\z}_k\}_{k\in \mathcal{K}}, \{\y_k\}_{k\in \mathcal{K}}):= & f_0(\x_0) + \sum_{k \in \mathcal{K}} \left(f_k(\x_k)	+\langle \lmd_k, {\z}_k \rangle + \frac{\beta}{2}\|{\z}_k\|^2 \right) + \notag \\
	&  \sum_{k \in \mathcal{K}} \left(\langle \y_k, \x_0 - \x^{\text{base}}_k+ {\z}_k \rangle + \frac{\rho}{2}\|\x_0 - \x^{\text{base}}_k+ {\z}_k\|^2 \right).
\end{align}
In iteration $(t+1)$, ADMM performs a Gauss-Seidel update on the primal variables $\x_0$,  $\{\x_k, (\x^{\text{base}}_k)\}_{k\in \mathcal{K}}$, $\{\z_k\}_{k\in \mathcal{K}}$, and the dual variables$\{\y_k\}_{k\in \mathcal{K}}$ in the augmented Lagrangian function $L_\rho$. All updates in $\{\x_k, (\x^{\text{base}}_k)\}_{k\in \mathcal{K}}$ and $\{\z_k\}_{k\in \mathcal{K}}$ can be distributed over contingencies. 

Finally, upon termination of ADMM, the returned solution $(\bar{\x}_0, \{\bar{\x}_k, \bar{\x}^{\text{base}}_k\}_{k\in \mathcal{K}},\{\bar{\z}_k\}_{k\in \mathcal{K}})$ of the augmented Lagrangian relaxation \eqref{eq:three_block_relaxation} may not be feasible for the original SC-OPF problem, i.e., $\z_k$ may not be zero. 
In order to drive the slack variables $\{\z_k\}_{k\in \mathcal{K}}$ to zero, we further update $\{\lmd_k\}_{k\in \mathcal{K}}$ and $\beta >0$ as in the classic augmented Lagrangian method, and restart ADMM to solve a new augmented Lagrangian relaxation. Consequently, we have a two-level ADMM summarized in Algorithm \ref{alg:two_level}. The inner-level index starts with $t=0$ while the outer-level index starts with $r=1$. This is due to the ways that inner and outer iteration complexities are calculated, as we will see in the next subsection. 
Observe that we explicitly project the dual iterate $\lmd_k^r+\beta_r\z_k^r$ onto some predetermined hypercube with lower bounds $\underline{\lmd}_k$ and upper bounds $\overline{\lmd}_k$ for $k\in \mathcal{K}$ to obtain $\lmd_k^{r+1}$. This projection operator is denoted by  $\Pi_{[\underline{\lmd}_k, \overline{\lmd}_k]}$. Such explicit bounds on $\lmd_k^{r+1}$ are standard for global convergence analysis of general nonconvex problems, e.g., see \cite{andreani2008augmentedcpld,andreani2008augmented}, and are indispensable for deriving theoretical iteration estimates. In principle, the hypercube should be large enough to contain dual multipliers corresponding to KKT solutions of the original problem. As we observed numerically, explicit projection steps are not necessary, and we think such behaviors might be related to local convergence properties of ALM.

\begin{algorithm}[h!]
	\caption{: A two-level ADMM for SC-ACOPF}\label{alg:two_level}
	\begin{algorithmic}[1]
		\STATE {Initialize} $\underline{\lmd}_k\leq \lmd_k^0\leq \overline{\lmd}_k$ for all $k\in \mathcal{K}$; $\beta_0, c, \tau> 1$; $\beta_1= \beta_0 c$; $r\gets 1$;
		\WHILE{outer stopping criteria is not satisfied}
		\STATE initialize $\beta \gets \beta_r $, $\rho \gets \tau \beta$, and $\x_0\in X_0$; $\lmd_k \gets \lmd^r_k$, $(\x^0_k, (\x^{\text{base}}_k)^{0}) \in X_k$, and $(\z_k^0, \y_k^0)$ such that $\lmd_k+\beta \z_k^0+\y_k^0= 0$ for all $k\in \mathcal{K}$; $t\gets 0$;
		\WHILE{inner stopping criteria is not satisfied}
        \STATE perform the following updates:
        \begin{subequations}\label{eq:distributed_admm}
        \begin{align}
        	\x_0^{t+1} = & \argmin_{ \x_0 \in X_0} f_0(\x_0) + \sum_{k\in \mathcal{K}} \left(\langle \y_k^t, \x_0 \rangle + \frac{\rho}{2}\| \x_0 - (\x^{\text{base}}_k)^t + \z^t_k\|^2\right);\label{eq:distributed_admm_first}\\
	        (\x^{t+1}_k, (\x^{\text{base}}_k)^{t+1}) = & \argmin_{ (\x_k,\x^{\text{base}}_k) \in X_k} f_k(\x_k) - \langle \y_k^t, \x^{\text{base}}_k \rangle + \frac{\rho}{2}\| \x^{t+1}_0 - \x^{\text{base}}_k + \z^t_k\|^2, \ \forall k \in \mathcal{K}; \label{eq:distributed_admm_second}\\
        	\z^{t+1}_k = &\frac{1}{\beta + \rho}\left( \rho((\x^{\text{base}}_k)^{t+1} - \x^{t+1}_0) - \lmd_k - \y_k^t\right), \ \forall k \in \mathcal{K};\label{eq:distributed_admm_third} \\
        	\y^{t+1}_k = & \y^{t}_k + \rho\left(\x^{t+1}_0 - (\x^{\text{base}}_k)^{t+1}+ \z^{t+1}_k\right), \ \forall k \in \mathcal{K};\label{eq:centralized_admm_dual}
    	\end{align}
        \end{subequations}
        $t\gets t+1$;
		\ENDWHILE
		\STATE denote the solution returned by the inner loop as $(\x^{r}_0, \{ \x^{r}_k, (\x^{\text{base}}_k)^{r}\}_{k\in \mathcal{K}},\{\z^{r}_k\}_{k\in \mathcal{K}})$;
        \STATE update $\lmd^{r+1}_k = \Pi_{[\underline{\lmd}_k,\overline{\lmd}_k]}(\lmd^{r}_k + \beta_r \z^{r}_k),\ \forall k\in \mathcal{K}$, and $\beta_{r+1}=\beta_0 c^{r+1}$; $r\gets r+1$;
		\ENDWHILE
	\end{algorithmic}
\end{algorithm}

Although we consider all contingencies in $\mathcal{K}$ in Algorithm \ref{alg:two_level}, this can be impractical due to the huge number of contingencies or limited computing resources. In contrast, we often run the two-level ADMM with only a subset of contingencies $\mathcal{K}'\subseteq \mathcal{K}$ that are potentially more severe than others. We will discuss how to select those contingencies in Section \ref{sec:ranking}.

Algorithm \ref{alg:two_level} is developed upon prior works \citep{sun2019two, sun2021two}, where single-phased nonlinear network problems are considered. In addition to problem scales, the SC-ACOPF is theoretically more challenging: in the compact formulation  formulation \eqref{eq:two_block_scopf}, both blocks of variables carry nonconvex OPF constraints, and hence the analysis from \citep{sun2019two, sun2021two} cannot be directly applied. 
We note that Algorithm \ref{alg:two_level} presents a rather conceptual algorithmic framework, while various specific issues, such as the qualities of solutions $\x_0^{t+1}$ and $\{\x_k^{t+1}, ((\x)^{\text{base}}_{k})^{t+1}\}_{k\in\mathcal{K}}$ and  termination criteria of the inner-level ADMM, need to be addressed in practice. We present a specific set of assumptions and the corresponding convergence results of the two-level ADMM in the next subsection.

\subsection{Convergence of the Two-level ADMM for SC-ACOPF}
In this subsection, we aim to establish certain global convergence properties of the two-level ADMM algorithm under some technical assumptions. Here, ``global convergence" refers to convergence to an approximate stationary solution from an arbitrary initial point, and should not be confused with convergence to a globally optimal solution.
The presence of nonconvexity in both stages brings challenges in the analysis. To deal with these challenges, we impose some mild conditions on the iterates of the algorithm. The results presented next aim to provide some theoretical supports to the two-level ADMM and help us understand its empirical performance.

We first define an approximate stationary point for the problem \eqref{eq:two_block_scopf} and its ALR problem \eqref{eq:three_block_relaxation}. 
We use $\partial f(\cdot)$ to denote the general subdifferential of a proper lower semi-continuous function $f:\R^n\rightarrow \R\cup\{+\infty\}$ \cite[Definition 8.3]{rockafellar2009variational}, and $\mathcal{N}_{C}(x)$ to denote the general normal cone of some closed set $C\subseteq \R^n$ at $x\in C$ \cite[Definition 6.3]{rockafellar2009variational}.
These two notations are defined in terms of the limiting behaviors of $f$ and $C$ in a neighborhood of the point of interest, and they both possess some useful properties (especially closedness) and calculus rules. We refer interested readers to the classic work on variational analysis by \cite{rockafellar2009variational}. For the purpose of this paper, it suffices to consider them as reasonable generalizations of their convex counterparts in a nonconvex and nonsmooth setting; indeed, they reduce to the standard notations under convexity.

\begin{definition}\label{def:scopf_approx_solution}
	Given $\epsilon>0$, we say $(\x_0, \{\x_k, \x_k^{\text{base}}\}_{k\in \mathcal{K}})$ is an $\epsilon$-stationary solution of problem \eqref{eq:two_block_scopf} if there exist $\{\y_k\}_{k\in \mathcal{K}}$ and 
        \begin{subequations}
            \begin{align}
                \d_0 & \in  \partial f_0(\x_0) +\sum_{k\in \mathcal{K}}\y_k + \mathcal{N}_{X_0}( \x_0 ),  \label{eq:base_case_dual_res}\\
                \d_k & \in  \begin{bmatrix} \tilde{\nabla} f_k(\x_k)\\-\y_k\end{bmatrix}+ \mathcal{N}_{X_k}(\x_k, \x_k^{\text{base}}),\ \text{where} \ \tilde{\nabla} f_k(\x_k) \in \partial f_k (\x_k), \quad \forall k\in \mathcal{K}, \label{eq:ctg_dual_res_main}\\
                \r_k & :=  \x_0 - \x_k^{\text{base}}, \quad \forall k\in \mathcal{K},
            \end{align}
	\end{subequations}
	such that $\max \left\{ \|\d_0\|, \|\d_1\|, \cdots, \|\d_{|\mathcal{K}|}\|,\|\r_1\|, \cdots, \|\r_{|\mathcal{K}|}\| \right \}\leq \epsilon.$
\end{definition}
\begin{definition}\label{def:alr_approx_solution}
	Given $\epsilon>0$,  $\{\lmd_k\}_{k\in \mathcal{K}}$, and $\beta >0$, we say $(\x_0, \{\x_k, \x_k^{\text{base}}\}_{k\in \mathcal{K}}, \{\z_k\}_{k\in \mathcal{K}})$ is an $\epsilon$-stationary solution of problem \eqref{eq:three_block_relaxation} if there exist $\{\y_k\}_{k\in \mathcal{K}}$ and $\d_0$ satisfying \eqref{eq:base_case_dual_res}, $\{\d_k\}_{k\in \mathcal{K}}$ satisfying \eqref{eq:ctg_dual_res_main}, and 
	\begin{subequations}
	\begin{align}
		0 & =  \lmd_k + \beta \z_k + \y_k, \quad \forall k\in \mathcal{K},  \\ 
		\s_k & :=  \x_0 - \x_k^{\text{base}} + \z_k, \quad \forall k\in \mathcal{K},
	\end{align}
	\end{subequations}
	such that $\max \left\{ \|\d_0\|, \|\d_1\|, \cdots, \|\d_{|\mathcal{K}|}\|,\|\s_1\|, \cdots, \|\s_{|\mathcal{K}|}\| \right \}\leq \epsilon$.
\end{definition}
The above two definitions generalize the standard KKT conditions for problems \eqref{eq:two_block_scopf} and \eqref{eq:three_block_relaxation} in the presence of nondifferetiable $f$ and implicit constraints $X_k$'s.
We first present the convergence of the inner-level ADMM to an approximate solution of \eqref{eq:three_block_relaxation}.
\begin{theorem}\label{thm:admm_iter_complex}
	Let $\lmd=\{\lmd_k\}_{k\in \mathcal{K}}$ and $\beta>0$ be given, and $\epsilon >0$. Suppose ADMM generates  iterates $\{\x^t_0,\{\x^t_k, (\x^{\text{base}}_0)^{t}\}_{k\in \mathcal{K}}, \{\z^t_k\}_{k\in \mathcal{K}} \}_{t\in \N}$ that satisfy the following conditions:
	\begin{itemize}
		\item[-] (descent in base case update) $\x_0^{t+1}$ is a stationary point of the optimization problem in \eqref{eq:distributed_admm_first} satisfying 
		\begin{align}
			& f_0(\x^{t+1}_0) + \sum_{k\in \mathcal{K}} \left(\langle \y_k^t, \x^{t+1}_0 \rangle + \frac{\rho}{2}\| \x^{t+1}_0 - (\x^{\text{base}}_k)^t + \z^t_k\|^2\right) \notag \\
			\leq & \  f_0(\x^{t}_0) + \sum_{k\in \mathcal{K}} \left(\langle \y_k^t, \x^{t}_0 \rangle + \frac{\rho}{2}\| \x^{t}_0 - (\x^{\text{base}}_k)^t + \z^t_k\|^2\right);\label{eq:descent_assumption_base}
		\end{align} 
		\item[-] (descent in contingency update) there exists a positive constant $\gamma >0$ such that for all $k \in \mathcal{K}$, $\x^{t+1}_k$ is a stationary point of the optimization problem in \eqref{eq:distributed_admm_second} satisfying
		\begin{align}
			& f_k(\x^{t+1}_k) - \langle \y_k^t, (\x^{\text{base}}_k)^{t+1} \rangle + \frac{\rho}{2}\| \x^{t+1}_0 - (\x^{\text{base}}_k)^{t+1} + \z^t_k\|^2 \notag \\
			\leq & \ f_k(\x^{t}_k) - \langle \y_k^t, (\x^{\text{base}}_k)^{t} \rangle + \frac{\rho}{2}\| \x^{t+1}_0 - (\x^{\text{base}}_k)^{t} + \z^t_k\|^2 - \gamma\beta \|(\x^{\text{base}}_k)^{t+1}-(\x^{\text{base}}_k)^{t}\|^2. \label{eq:descent_assumption_ctg}
		\end{align}
	\end{itemize}
	Then ADMM finds an $\epsilon$-stationary solution of the ALR problem \eqref{eq:three_block_relaxation} in the sense of definition \ref{def:alr_approx_solution} in at most
	\begin{align}\label{eq:admm_iter_complex}
		T\leq \left\lceil  \frac{ 2\rho^2 |\mathcal{K}| (\overline{L}(\lmd,\beta)-\underline{L}(\lmd,\beta))}{\min \{\gamma\beta, (\beta+\rho)/2-\beta^2/\rho \} }\cdot\frac{1}{\epsilon^2 } \right\rceil	
	\end{align}
	iterations, where
	\begin{align}
		\overline{L}(\lmd,\beta) := & f_0(\x^0_0) + \sum_{k \in \mathcal{K}} \left(f_k(\x^0_k)	+\langle \lmd_k, \z^0_k \rangle + \frac{\beta}{2}\|\z^0_k\|^2 \right) \notag \\
	& + \sum_{k \in \mathcal{K}} \left(\langle y^0_k, \x^0_0 - (\x^{\text{base}}_k)^0+ {\z}^0_k \rangle + \frac{\rho}{2}\|\x^0_0 - (\x^{\text{base}}_k)^0+ {\z}^0_k\|^2 \right), \\
		\underline{L}(\lmd,\beta):= & \min \left\{\sum_{g\in \mathcal{G}} c_g(p_{g0}):\ p_{g0}\in [\underline{p}_g, \overline{p}_g], \ \forall g\in \mathcal{G} \right\} - \frac{\|\lmd\|^2}{\beta}.
	\end{align}
\end{theorem}

The proof of Theorem \ref{thm:admm_iter_complex} is provided in \ref{sec:proof_of_admm}. 

Next we present the convergence of the overall two-level ADMM Algorithm \ref{alg:two_level}. 
\begin{theorem}\label{thm:alm_iter_complex}
	Let $\epsilon >0$. Suppose the assumptions in Theorem \ref{thm:admm_iter_complex} hold, and each ADMM returns an $\epsilon$-stationary solution of \eqref{eq:three_block_relaxation} in the sense of Definition \ref{def:alr_approx_solution}. 
	Moreover, assume there exists some $0 < \overline{L} < +\infty$ such that $\overline{L}\geq \overline{L}(\lmd^r, \beta_r)$ for all $r\in \N$. Define constants
	\begin{itemize}
		\item[] $\eta := \min\{\gamma, \frac{\tau+1}{2} - \frac{1}{\tau}\}$,
		\item[] $\Lambda := \max \left\{ \sum_{k\in \mathcal{K}} \|\lmd_k\|^2 : \ \lmd_k \in [\underline{\lmd}_k, \overline{\lmd}_k], \ \forall k\in \mathcal{K}\right\}$,
		\item[] $\underline{L} :=  \min \left\{\sum_{g\in \mathcal{G}} c_g(p_{g0}):\ p_{g0}\in [\underline{p}_g, \overline{p}_g], \ \forall g\in \mathcal{G} \right\}- \frac{\Lambda}{\beta_0}$, and $\delta_L :=\overline{L}- \underline{L}$.
	\end{itemize}
	Then the two-level ADMM (Algorithm \ref{alg:two_level}) finds an $\epsilon$-stationary solution of problem \eqref{eq:two_block_scopf} in at most $R$ outer-level updates, where
	\begin{align}\label{eq:outer_level_bd}
		R \leq \left \lceil   \log_c \left( \frac{4 \delta_L}{\beta_0}\cdot\frac{1}{\epsilon^2} \right)\right \rceil,
	\end{align}
	 and a total of 
	\begin{align}\label{eq:total_iter_complex}
		\left\lceil \left( \frac{2c\tau^2 \beta_0|\mathcal{K}|\delta_L}{(c-1) \eta}\right)\left( \frac{c^R-1}{\epsilon^2} \right)\right \rceil + R	=\mathcal{O}\left(\frac{1}{\epsilon^{4}}\right)
	\end{align}
	inner-level ADMM iterations.
\end{theorem}
The proof of Theorem \ref{thm:alm_iter_complex} is provided in \ref{sec:proof_of_alm}. We note that Algorithm \ref{alg:two_level} provides a convenient update scheme that simplifies the iteration complexity analysis in Theorem \ref{thm:alm_iter_complex}, while some other update schemes can perform better in practice. For example, the explicit projection onto the hypercube $[\underline{\lmd}_k,\overline{\lmd}_k]$ may not always be necessary, and we can simply set $\lmd^{r+1}_k = \lmd^{r}_k + \beta_r \z^{r}_k$ for $k\in \mathcal{K}$ as in the classic augmented Lagrangian method with partial elimination of constraints, where local convergence results have been well studied in \cite{bertsekas2014constrained}. In addition, a geometrically increasing sequence of penalty parameters, i.e., $\beta_{r+1} = \beta_0 c^{r+1}$, may be deemed too aggressive for some instances, and in practice we can keep the penalty unchanged, i.e., $\beta_{r+1}=\beta_r$, as long as a sufficient amount of  decrease of infeasibility is observed.
We acknowledge that the assumptions upon which Theorems \ref{thm:admm_iter_complex} and \ref{thm:alm_iter_complex} are established may not be satisfied for all SC-ACOPF instances. Our results in this section aim to provide a theoretical support for this framework. 

%% file: 6ContingencyScreening.tex
\section{Contingency Solution Strategy}\label{sec:ranking}

\subsection{Recourse Models}
For each contingency \(k\in\calK\), we solve the following recourse model
\begin{align}\label{eq:RecourseModel}
    \min\quad&c_k^\sigma\quad
    \mathrm{s.t.}\quad\eqref{eq:scopf_cost}-\eqref{eq:scopf_slackbounds}.
\end{align}
Due to the presence of the constraints~\eqref{eq:scopf_realpowercntg} and~\eqref{eq:scopf_PVPQswitch}, we are not able to directly apply an interior point solver to the problem~\eqref{eq:RecourseModel} as mentioned in Section~\ref{sec:smoothing}.
Instead, we first solve the following smoothed recourse model
\begin{align}\label{eq:SmoothedRecourseModel}
    \min\quad&c_k^\sigma\quad
    \mathrm{s.t.}\quad\eqref{eq:scopf_cost}-\eqref{eq:scopf_txpowerlimit_destination},\eqref{eq:scopf_bounds_vpq}-\eqref{eq:scopf_slackbounds},\eqref{eq:SmoothedRealPowerDisjunction},\text{ and }\eqref{eq:SmoothedConstraintApproximation}.
\end{align}
One potential issue with the smoothed recourse model is that if the approximation parameter \(\epsilon>0\) is not small enough, the obtained solution can be infeasible as the solution values \((\hat{p}_{gk},\hat{q}_{gk},\hat{v}_{i_gk})_{g\in G_k}\) and \(\hat{\Delta}_k\) violate the disjunctive constraints~\eqref{eq:scopf_realpowercntg} and~\eqref{eq:scopf_PVPQswitch}.
When the infeasibility occurs, we can restrict the domain of some active power and voltage variables and resolve the recourse model~\eqref{eq:RecourseModel} so that the disjunctive constraints~\eqref{eq:scopf_realpowercntg} and~\eqref{eq:scopf_PVPQswitch} become smooth on the restricted domain.
More specifically, given a constraint violation threshold \(\mu>0\), we define the sets of generators \(\hat{G}_k^{\rm p,-}:=\{g\in G_k:p_{g0}+\alpha_g\hat{\Delta}_k\le \underline{p}_g+\mu,\;\alpha_g>0\}\), \(\hat{G}_k^{\rm p,+}:=\{g\in G_k:p_{g0}+\alpha_g\hat{\Delta}_k\ge \bar{p}_g-\mu,\;\alpha_g>0\}\), \(\hat{G}_k^{\rm q,-}:=\{g\in G_k:\hat{q}_{gk}\le \underline{q}_g+\mu\}\), and \(\hat{G}_k^{\rm q,+}:=\{g\in G_k:\hat{q}_{gk}\ge \bar{q}_g-\mu\}\).
The restricted recourse model is defined to be
\begin{align}\label{eq:RestrictedRecourseModel}
    \min\quad c_k^\sigma\quad
    \mathrm{s.t.}\quad&\eqref{eq:scopf_cost}-\eqref{eq:scopf_slackbounds}\\
    &p_{gk}=\underline{p}_g,\;\Delta_k\le(\underline{p}_g-p_{g0})/\alpha_g,\;\forall\,g\in\hat{G}_k^{\rm p,-},\notag\\
    &p_{gk}=\bar{p}_g,\;\Delta_k\ge(\bar{p}_g-p_{g0})/\alpha_g,\;\forall\,g\in\hat{G}_k^{\rm p,+},\notag\\
    &q_{gk}=\underline{q}_g,\;v_{i_gk}\ge v_{i_g0},\;\forall\,g\in\hat{G}_k^{\rm q,+},\notag\\
    &q_{gk}=\bar{q}_g,\;\;v_{i_gk}\le v_{i_g0},\;\forall\,g\in\hat{G}_k^{\rm q,-},\notag\\
    &v_{i_gk}=v_{i_g0},\;\forall\,g\in G_k\setminus(\hat{G}_k^{\rm q,+}\cup\hat{G}_k^{\rm q,-}).\notag
\end{align}
We summarize the recourse problem solution strategy in Algorithm~\ref{alg:recourse}.
\begin{algorithm}[h!]
	\caption{: Recourse Model Solution Strategy for Contingency \(k\in\calK\)}\label{alg:recourse}
	\begin{algorithmic}[1]
        \STATE Solve the smoothed recourse model~\eqref{eq:SmoothedRecourseModel}
        \IF{violation of disjunctive constraints~\eqref{eq:scopf_realpowercntg} and~\eqref{eq:scopf_PVPQswitch} \(>\mu\)}
        \STATE Solve the restricted recourse model~\eqref{eq:RestrictedRecourseModel} and update the solutions
        \ENDIF
	\end{algorithmic}
\end{algorithm}

\subsection{Contingency Screening}
Ideally, all contingencies in $\calK$ should be exhaustively considered in the SC-ACOPF problem. However, this can be extremely time consuming when $|\calK|$ is large. 
Motivated by this challenge, in this section we present a method for ranking the contingencies in the list $\calK$ according to their severity, with the most severe ones at the top of the list. 
The proposed contingency ranking is based on a given base case solution (corresponding to $k=0$). Indeed, we estimate the severity of a contingency $k\in{\calK}$ by approximating the penalty cost $c_k^\sigma$ defined in \eqref{eq:scopf_cost}, given that the base case solution is substituted into the constraints \eqref{eq:scopf_realpower_destination}-\eqref{eq:scopf_slackbounds}.
%
%
%

Specifically, let $\tilde{k}\in\calK$ be a generator contingency associated with the outage of generator $\tilde{g}\in \calG \setminus G_{\tilde{k}}$ located at bus ${\tilde{i}}$, and suppose this generator was producing active and reactive powers $p_{\tilde{g}0}$ and $q_{\tilde{g}0}$ in the base case solution. In the sequel, without loss of generality, we assume the given base case solution satisfies constraints \eqref{eq:scopf_cost}-\eqref{eq:scopf_slackbounds} for $k=0$, with all slack variables $\sigma_{i0}^{P\pm}, \sigma_{i0}^{Q\pm}, \sigma_{e0}^S, \sigma_{f0}^S$ being equal to zero.
%
%
Observe that such a base case solution will also satisfy \eqref{eq:scopf_cost}-\eqref{eq:scopf_slackbounds} for $k=\tilde{k}$ with positive slack variables (see \eqref{eq:scopf_realpower_nodal} and \eqref{eq:scopf_reactivepower_nodal})
\begin{align}
\sigma_{\tilde{i}\tilde{k}}^{P-} = \max \{0, p_{\tilde{g}0} \},\;
\sigma_{\tilde{i}\tilde{k}}^{P+} = -\min \{0, p_{\tilde{g}0} \},\;
\sigma_{\tilde{i}\tilde{k}}^{Q-} = \max \{0, q_{\tilde{g}0} \},\;
\sigma_{\tilde{i}\tilde{k}}^{Q+} = -\min \{0, q_{\tilde{g}0} \}. 
\end{align}
Accordingly, the severity criterion, i.e., the penalty cost \eqref{eq:scopf_cost} of a generator contingency $\tilde{k}$ can be approximated as:
\begin{align}
 \tilde{c}_{\tilde{k}}^\sigma :=
 \sum_{i\in\calI} \left(c_{i\tilde{k}}^p(\sigma_{i\tilde{k}}^{P+}+\sigma_{i\tilde{k}}^{P-})+c_{i\tilde{k}}^q(\sigma_{i\tilde{k}}^{Q+}+\sigma_{i\tilde{k}}^{Q-})\right) = 
  c^p_{{\tilde{i}}\tilde{k}} (|p_{\tilde{g}0}|) +c^q_{{\tilde{i}}\tilde{k}} (|q_{\tilde{g}0}|).
\end{align} 
Recall that $c_{ik}^p(\cdot)$ and $c_{ik}^q(\cdot)$ are convex pwl increasing functions. With a similar argument, suppose $\tilde{k}\in\calK$ is a line contingency associated with the outage of line $\tilde{e}\in\calE\setminus\calE_{\tilde{k}}$, and let $p^o_{\tilde{e}0}$ and $q^o_{\tilde{e}0}$ ($p^d_{\tilde{e}0}$ and $q^d_{\tilde{e}0}$, resp.) be the active and reactive power flowing into this line at its origin (destination, resp.) in the base case. Analogously, the given base case solution satisfies \eqref{eq:scopf_cost}-\eqref{eq:scopf_slackbounds} for $k=\tilde{k}$ with positive slack variables in equations \eqref{eq:scopf_realpower_nodal} and \eqref{eq:scopf_reactivepower_nodal}. Therefore, 
\begin{align}
\tilde{c}_{\tilde{k}}^\sigma :=
c^p_{{ i^o_{\tilde{e}} } \tilde{k}} (|p^o_{\tilde{e}0}|) 
+ c^p_{{ i^d_{\tilde{e}} } \tilde{k}} (|p^d_{\tilde{e}0}|) 
+c^q_{{ i^o_{\tilde{e}} }\tilde{k}} (|q^o_{\tilde{e}0}|)
+ c^q_{{ i^d_{\tilde{e}} }\tilde{k}} (|q^d_{\tilde{e}0}|).
\end{align}
The penalty cost of a transformer contingency can be approximated likewise. Now, given a base case solution, we compute the severity index $\tilde{c}_{{k}}^\sigma$ for all $k\in\calK$. Then, contingency ranking is performed by ordering $\tilde{c}_{{k}}^\sigma$ from the greatest to the least.
%
%
%
%
%
The proposed contingency ranking method is summarized in Algorithm \ref{alg:ctg_ranking} and is used to give priority to high-impact contingencies in the SC-ACOPF problem. This will be discussed in Section \ref{sec:Other Algorithmic Development}.

\begin{algorithm}[h]
	\caption{: Contingency Ranking}\label{alg:ctg_ranking}
	\begin{algorithmic}[1]
		\STATE {\textbf{Input}:} base case generations and flows $\{ \{r_{{g}0},r^o_{{e}0}, r^d_{{e}0}, r^o_{{f}0}, r^d_{{f}0} \} \: : \:  \forall {g}\in\calG, {e}\in\calE, {f}\in\calF, r\in\{p,q\} \}$;
		\STATE {\textbf{Output}:} sorted contingency list $\calK$;
		\STATE {Initialize}    $\tilde{c}_{{k}}^\sigma = 0 $ for all $k\in\calK$;
		\FOR{ $k \in \calK $}
		 \IF{ $\calG\setminus G_{k} \ne \emptyset$ } \STATE {Set $g\in\calG\setminus G_{k}$ and find $i\in\calI$ such that $G_{i0}=\{g\}$; 
		 	\STATE Update  $\tilde{c}_{{k}}^\sigma \gets	c^p_{{{i}}{k}} (|p_{{g}0}|) +c^q_{{{i}}{k}} (|q_{{g}0}|)$;
		 } \ELSIF{ $\calE\setminus\calE_{{k}} \ne \emptyset $ } \label{step_line_start}
	       \STATE{ Set $e\in\calE\setminus\calE_{{k}}$ and find $(i^o_e,i^d_e)$ such that $ E^o_{i^o_e 0}\cap E^d_{i^d_e 0}  = \{e\} $; }
	       \STATE{ Update $\tilde{c}_{{k}}^\sigma \gets    	c^p_{{ i^o_{{e}} } {k}} (|p^o_{{e}0}|) 
	       	+ c^p_{{ i^d_{{e}} } {k}} (|p^d_{{e}0}|) + c^q_{{ i^o_{{e}} } {k}} (|q^o_{ {e}0}|) + c^q_{{ i^d_{{e}} }{k}} (|q^d_{{e}0}|) $; } 
	       \ENDIF \label{step_line_end}
	     \STATE \textbf{Repeat} steps \ref{step_line_start} to \ref{step_line_end} for transformers;
		\ENDFOR
		\STATE sort $\calK$ according to value of $\tilde{c}_{{k}}^\sigma, k\in\calK$;
	\end{algorithmic}
\end{algorithm}

%% file: 7OtherImplementations.tex
\section{Parallel Implementation}
\label{sec:Other Algorithmic Development}

As the for-loop in Algorithm~\ref{alg:overview} (line~\ref{alg:overview:ContingencyForLoop}) and the update step~\eqref{eq:distributed_admm} in Algorithm~\ref{alg:two_level} can be executed in parallel, we are able to use the Message Passing Interface (MPI) to manage the communication between our manager process and numerous worker processes in a multi-node computing environment.
However, to timely update the solutions without delaying the algorithm execution on worker processes, we have designed a manager-worker-writer parallel implementation instead of the usual manager-worker implementation, as described below.

In the beginning, the manager process sends subproblem indices \(k\in\calK\) to each worker process and tells the writer process to output a default solution (i.e., the starting points in the subproblems).
Whenever a worker process finishes a subproblem, it sends the solution back to the manager process and receives a new subproblem index for it to solve.
The subproblem solution is then stored temporarily in the manager process.
After the writer process completes the output, it sends a signal to the manager process, which would then send all the stored subproblem solutions to the writer process.
In this way, the writer process would be able to continuously updating the solution files, which could potentially be of the sizes of tens of gigabytes, while the manager process is able to coordinate the worker processes at the same time.
The parallelization structure is illustrated in Figure~\ref{fig:parallel}.

\begin{figure}[htbp]
    \centering
    \includegraphics[width=.6\textwidth]{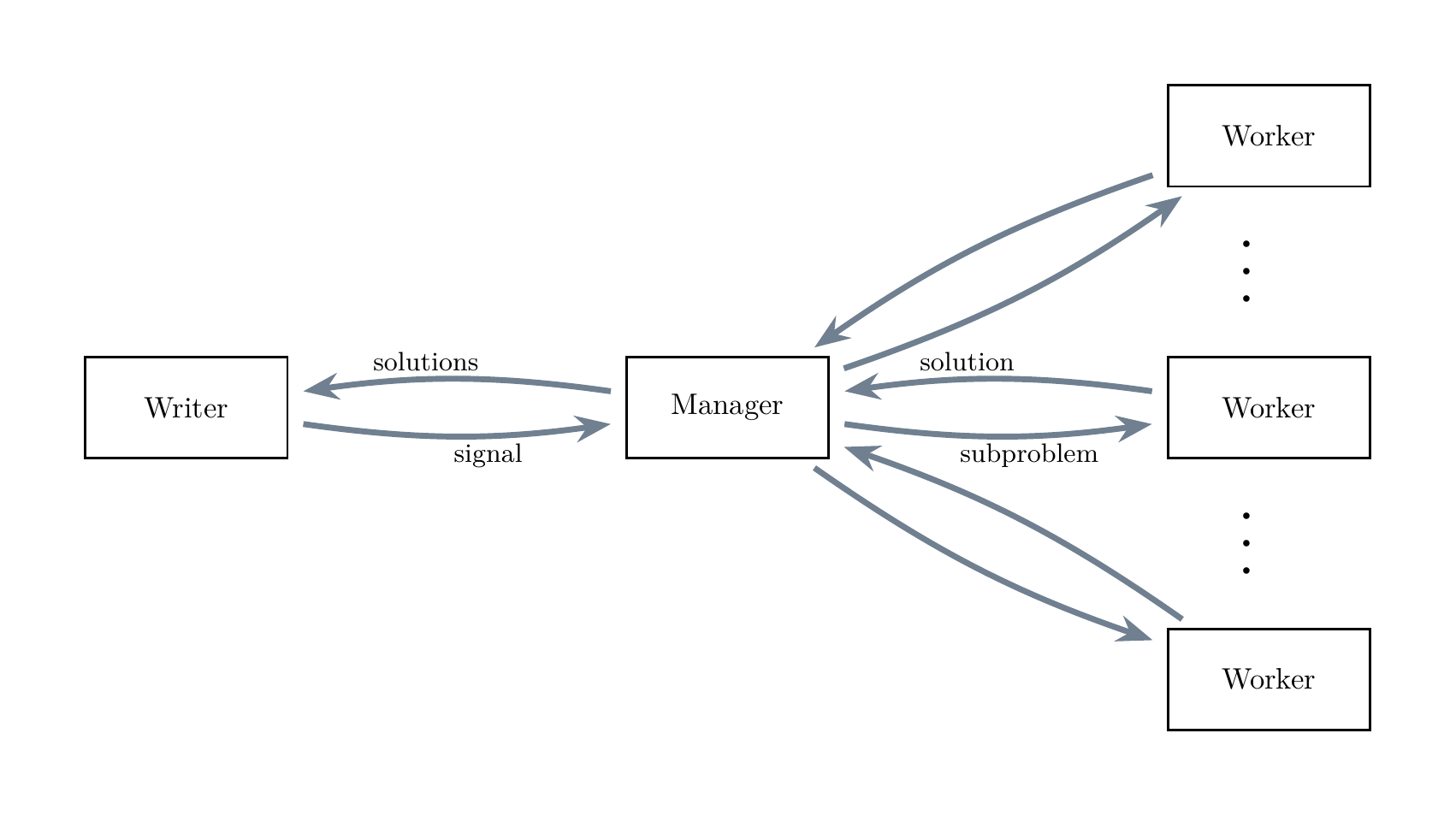}
    \caption{Illustration of Our Parallel Implementation}
    \label{fig:parallel}
\end{figure}


%% file: 8Computation.tex
\section{Computational Experiments and Results}
\label{sec:comp}
\subsection{Evaluation Framework and Platform Specifications}
\label{subsec: Evaluation Methods}
In order to carry out a fair evaluation of the proposed algorithm and make a valid comparison with the existing methods, we follow a two-phase assessment proposed by \cite{GOChallenge1_formulation}:
\begin{itemize}
    \item \textbf{Phase I}: The algorithm should report the base case solution $X_0$ within $45$ minutes.
    \item \textbf{Phase II}: The algorithm should report the solution of all contingencies $k\in{\calK}$, given the previously reported base case decision variables in Phase I. The time limit for this phase is $2$ seconds per contingency.
\end{itemize}
%
%
%
Absence of a base case solution from Phase I would directly cause the failure of Phase II evaluation.
We remark that our algorithmic framework is able to produce a base case solution within the first 10 minutes for most of the networks discussed in this section, thanks to the parallel implementation in Section~\ref{sec:Other Algorithmic Development}.

The experiments are conducted on a $6$-node cluster, where each node has $64$ GiB of 4-channel 2133 MHz DDR4 SDRAM memory, two Intel Xeon E$5$-$2670$ v$3$ (Haswell) CPUs, each with 12 cores (24 cores per node) and a clock speed of $2.30$ GHz.
The nonlinear optimization subproblems~\eqref{eq:distributed_admm} and~\eqref{eq:RestrictedRecourseModel} are solved using the solver Ipopt~\citep{wachter2006implementation} version 3.12, with the linear system solver configured to be \cite{HSL} MA57.


\subsection{Datasets and Model Parameters} \label{subsec: datasets}
We adopt a comprehensive dataset composed of $17$ network models which is available at \cite{GODataSet} and its characteristics are summarized in Table \ref{tab: network stat}. Each network model has $20$ different instances which differ in terms of number of components and operating conditions. When the number of components differs for various network instances, a range is specified in Table \ref{tab: network stat}.
Different instances of a network are completely independent and must be solved individually.
Hence, the dataset contains $17\times20=340$ network instances ranging from $500$-bus to $30,000$-bus systems. The number of contingencies $|\calK|$ also reaches $22,000$ in the largest instance.
%
%
%
Further data on penalty function and model parameters are provided in
\ref{sec:Weights_and_Penalty Function_Parameters}.

\subsection{Solution Quality and Scalability}
In Fig. \ref{fig: total obj vs benchmark}, the objective value of the SC-ACOPF problem \eqref{eq:scopf} obtained from our proposed methodology is compared with the ARPA-E benchmark algorithm, whose code is available at \cite{ARPAEBenchmark-Code} (the version used for benchmarking is kept in the directory \texttt{src/script}).
To have a fair comparison, both methodologies use the same platform and are subjected to the same time limit, i.e., $45$ minutes for Phase I and $2$ seconds/contingency for Phase II (see Section \ref{subsec: Evaluation Methods}). The proposed algorithm consistently outperforms ARPA-E benchmark algorithm in all $340$ network instances, thereby resulting in lower generation cost and penalty values. Notice that the objective axis is in the logarithmic scale.
    \begin{figure}[h!]
     	\includegraphics[height=4cm, keepaspectratio=true]{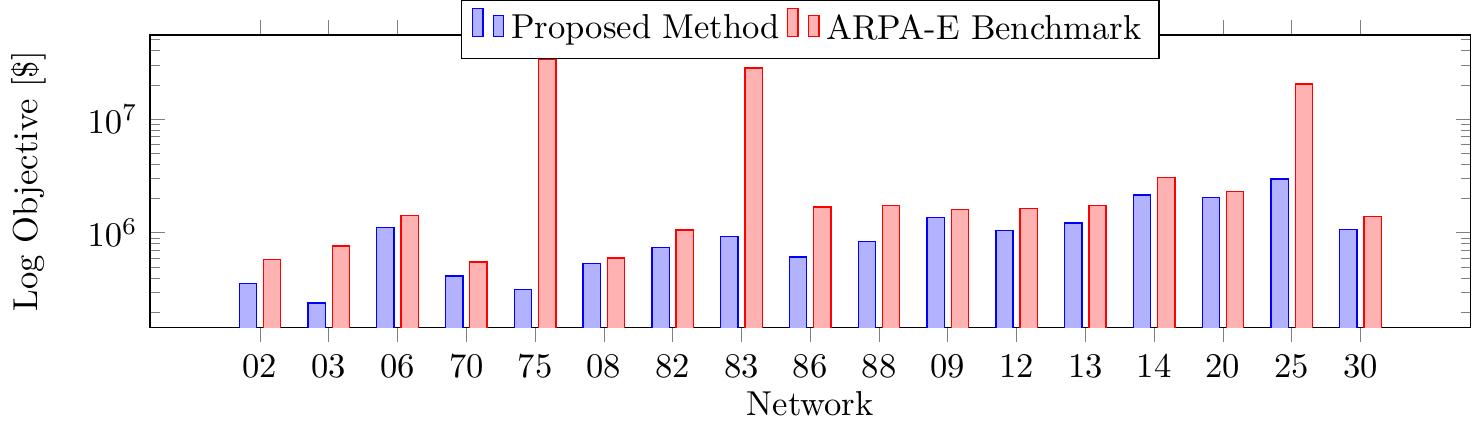}
     	\centering
     	\caption{Comparison between the proposed method and the ARPA-E benchmark method. The vertical axis shows the geometric mean of the objective values associated with all $20$ instances of each network model.}
     	\label{fig: total obj vs benchmark}
    \end{figure}
%
    \begin{figure}[h!]
     	\includegraphics[height=4cm, keepaspectratio=true]{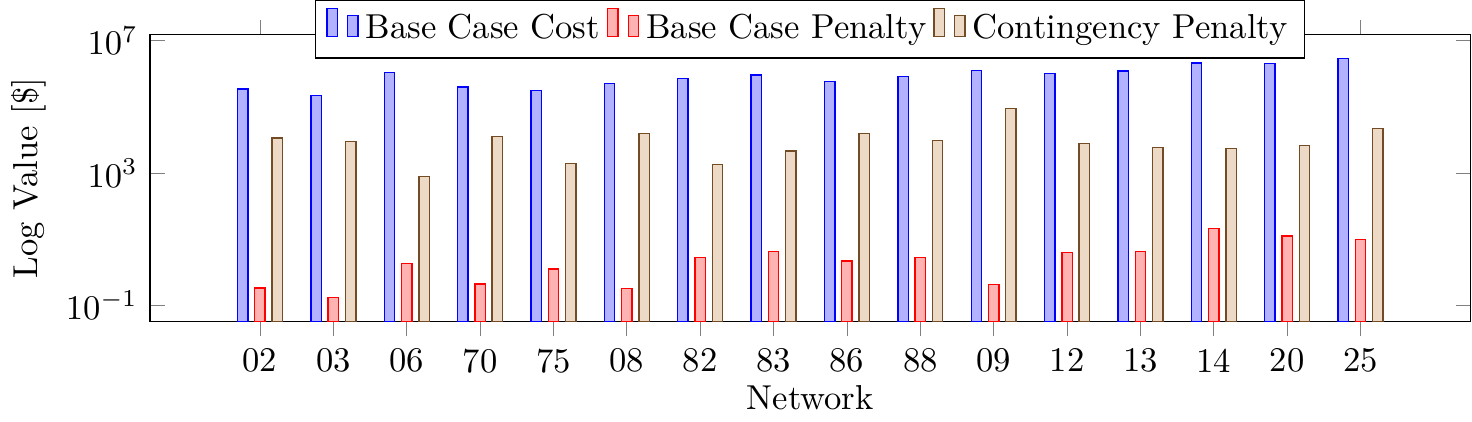}
     	\centering
     	\caption{Breakdown of the SC-ACOPF objective. The vertical axis has logarithmic scale and shows the geometric mean of the values associated with all $20$ instances of each network model.}
     	\label{fig: obj decomposition}
    \end{figure}

Next, Fig. \ref{fig: obj decomposition} provides a breakdown of the objective value \eqref{eq:scopf_obj} that results from three different terms, namely base case cost $\sum_{g\in\calG} c_g(p_{g0})$, base case penalty $c_0^\sigma$, and contingency penalty $({1}/{|\calK|})\sum_{k\in\calK}c_k^\sigma$. As can be observed, the generation cost of base case constitutes the major part of the total objective value. The base case solution is almost AC feasible as its penalty is negligible compared to the other two terms. The contingency penalty can be interpreted as the load shedding cost, which seems inevitable following the occurrence of (not necessarily all, but  some) contingencies in list $\calK$.
Figs. \ref{fig: total obj vs benchmark} and \ref{fig: obj decomposition} show that 
under the time limit of $45$ minutes for Phase I, the proposed Algorithm \ref{alg:overview} is scalable up to networks with $30$k buses and $22$k contingencies.

\subsection{Computation of the Two-Level ADMM}\label{subsec: admm}
In this subsection, we discuss the performance of the two-level ADMM algorithm, which is used to produce a base case solution for Phase I assessment. 
Due to the 45-minute time limit in Phase I, only a subset $\cal{K}'\subseteq \cal{K}$ of contingencies are selected to participate in the two-level ADMM algorithm. Let $n$ denote the number of contingencies solved sequentially on each core in every iteration: we set $n=20$ if $|\mathcal{I}|\in (0, 1000]$,  $n=15$ if $|\mathcal{I}|\in (1000, 5000]$, $n=10$ if $|\mathcal{I}|\in (5000, 10000]$, and $n=5$ if $|\mathcal{I}|\in (10000, +\infty)$. Then the total number of selected contingencies is set to $|\cal{K}'|=\min\{141 \times n, |\cal{K}|\}$ (here 141 is the number of our available worker processes in the cluster), while members of $\cal{K}'$ are determined by a call of Algorithm \ref{alg:ctg_ranking}. See \ref{sec:admm_parameters} for other implementation details. 
Fig. \ref{fig: admm iter} displays the average number of iterations of the two-level ADMM for all networks and the time at which the Phase I base case solution is outputted. Here, ``Inner Iteration" refers to the total number of updates \eqref{eq:distributed_admm_first}-\eqref{eq:centralized_admm_dual} in Algorithm \ref{alg:two_level}, while ``Outer Iteration" refers to the number of times the outer-level dual information $(\lambda, \beta)$ being updated. We terminate the two-level ADMM if the maximum constraint violation of the consensus constraint \eqref{eq:two_blocl_couple}, i.e., $\max_{k\in \cal{K}'} \{\|\x_0-\x^{\text{base}}_k\|_{\infty}\}$, is less than $10^{-4}$. Fig. \ref{fig: admm res} shows the geometric mean of this metric for each network at termination of the two-level ADMM. 

\begin{figure}[h!]
     	\includegraphics[height=4cm, keepaspectratio=true]{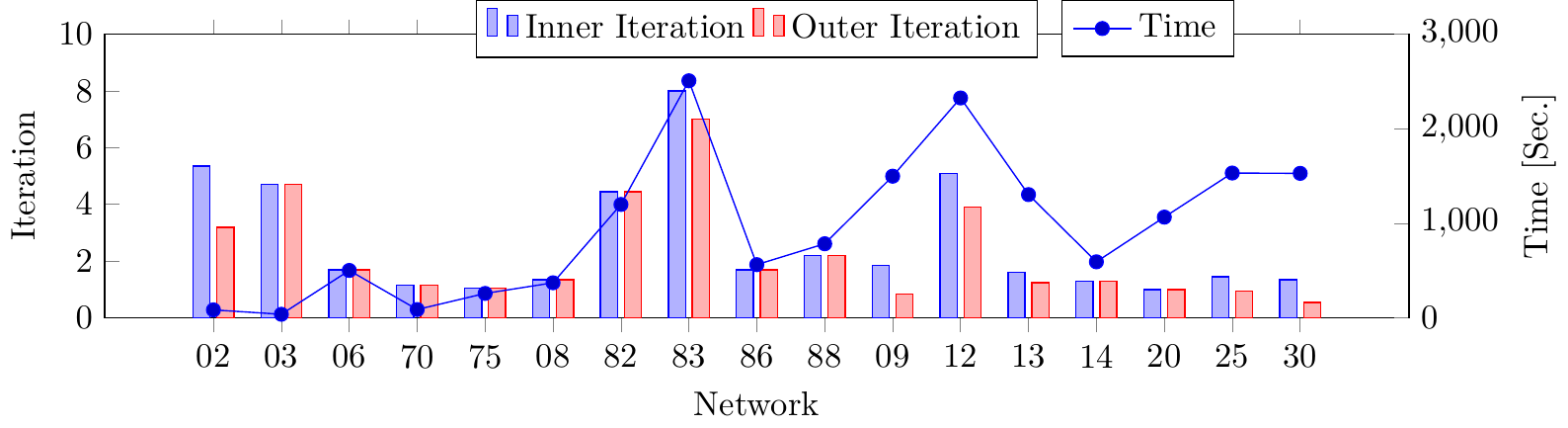}
     	\centering
     	\caption{Number of inner and outer level updates and the time used to produce the base solution. The vertical axis on the left shows the average number of inner/outer updates within the time limit, and the vertical axis on the right shows the average time to finalize the base case in Phase I.}
     	\label{fig: admm iter}
\end{figure}
\begin{figure}[h!]
     	\includegraphics[height=4cm, keepaspectratio=true]{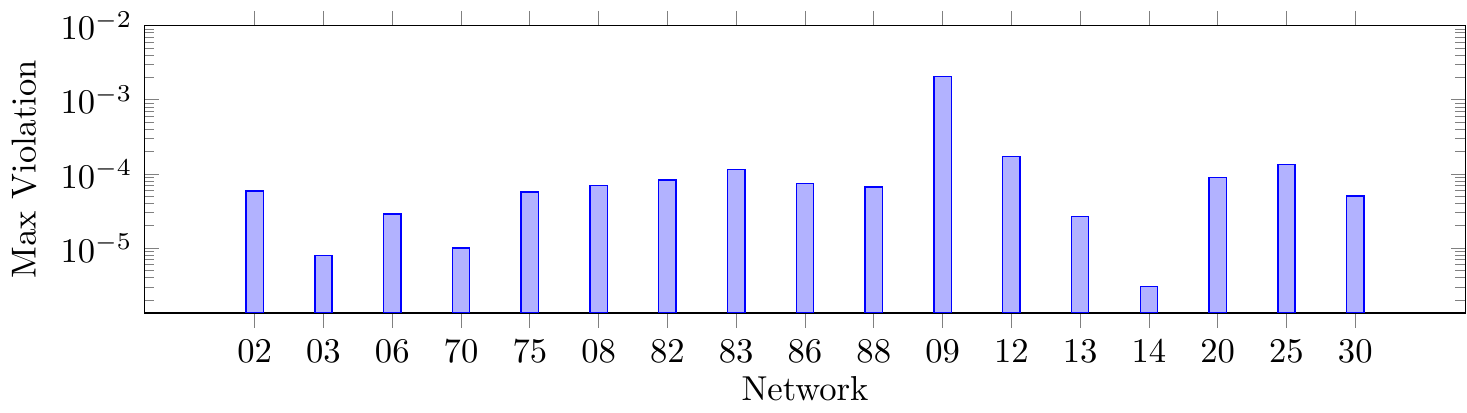}
     	\centering
     	\caption{Maximum violation of the consensus constraints between the base case and contingency solutions upon termination. The vertical axis shows the geometric mean of max constraint violation over scenarios of each network model.}
     	\label{fig: admm res}
\end{figure}

We first note that all cases are able to produce the base case solution before the Phase I time limit, which is highly desirable from a practical point view. For several networks of small or medium sizes, the two-level ADMM is able to perform up to eight rounds of updates of the base solution, which iteratively incorporates contingency effects into the base case solution through the augmented Lagrangian framework embedded in subproblem \eqref{eq:distributed_admm_first}. For most scenarios of large networks with more than 10k buses, the two-level ADMM is able to bring the max violation down to the order of $10^{-4}\sim10^{-6}$ in even 1 or 2 inner iterations. Such a behavior indicates that the proposed algorithm is able to locate feasible solutions in only a few iterations, and allows us to include a large number of contingencies in Phase I.


\subsection{Computational Efficiency and Contingency Ranking in the Recourse Model}
    
  
To further demonstrate the scalability of Algorithm \ref{alg:recourse}, Fig. \ref{fig: ctg_wall_time} depicts the average computation time needed to solve the recourse model.
As can be seen, the computation time is less than one minute for small up to medium-size networks with $10$k buses, and increases to around $250$ seconds for network $25$ which has around $24$k buses. Note that given $144$ available computation cores, the parallel computation time of each contingency is within $2$ seconds for Phase II (see Section \ref{subsec: Evaluation Methods}), as roughly speaking we have (250 sec$/$contingency)$\div$(144 cores)< 2 sec$/$(contingency$\cdot$core). 
%
%
    \begin{figure}[h!]
     	\includegraphics[height=4cm, keepaspectratio=true]{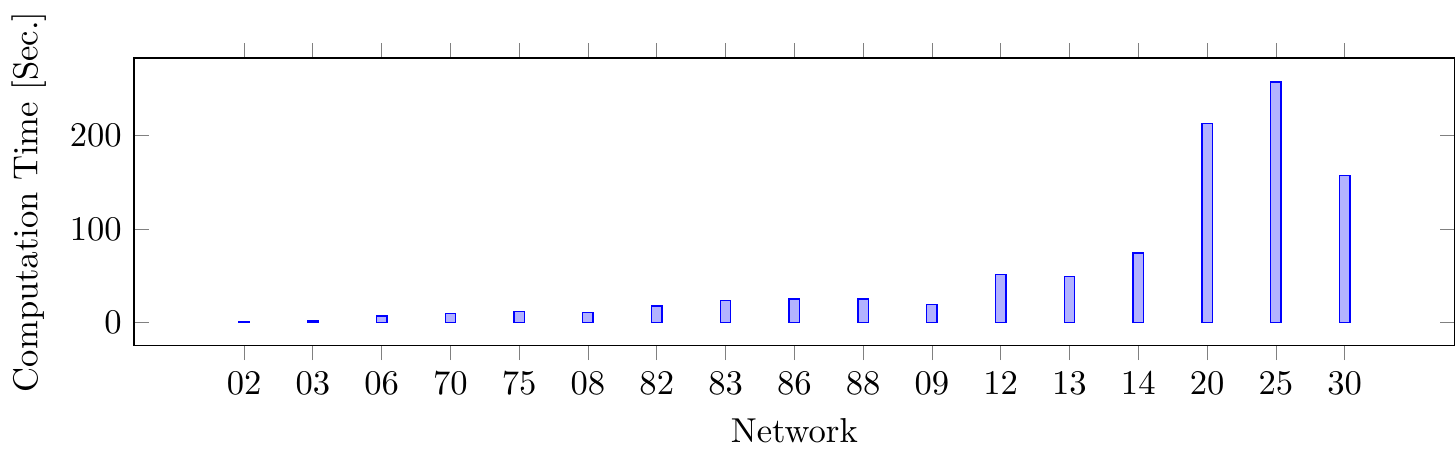}
     	\centering
     	\caption{Average wall-clock computation time for solving the recourse model. The vertical axis shows the geometric mean of the average computation time associated with all $20$ instances of each network model.}
     	\label{fig: ctg_wall_time}
     \end{figure}

Another interesting observation is on the efficacy of the ranking algorithm (Algorithm \ref{alg:ctg_ranking}) in identifying the most severe contingencies, which is shown in Fig. \ref{fig: ctg analysis}. Here the vertical axis is the improved penalty of each contingency $k$ in the sorted list $\calK$, i.e., $\tilde{c}_{{k}}^\sigma-{c}_{{k}}^\sigma$, where $\tilde{c}_{{k}}^\sigma$ is
the estimated penalty from Algorithm \ref{alg:ctg_ranking} 
and ${c}_{{k}}^\sigma$ is the objective value of the recourse model \eqref{eq:RecourseModel} returned by Algorithm \ref{alg:recourse}. This quantity shows how much penalty can be avoided if the system operator solves the recourse model \eqref{eq:RecourseModel} and takes corrective actions, instead of taking the prespecified base case solution as a preventive measure and therefore not responding to the contingency occurrence. As can be seen in Fig. \ref{fig: ctg analysis}, for the tail contingencies in the sorted list $\calK$, the system operator can barely gain benefits from solving the recourse model, so time and resources can be saved for the more important contingencies.
%
%


\begin{figure}[htb!]
\centering
\subfloat[network $02$]{%
\includegraphics[width=8cm]{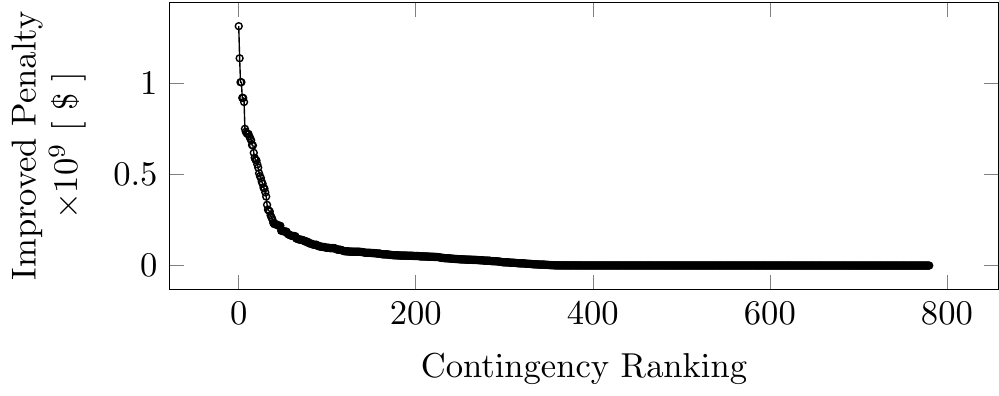}
}%
\hfill
\subfloat[network $06$]{%
\includegraphics[width=8cm]{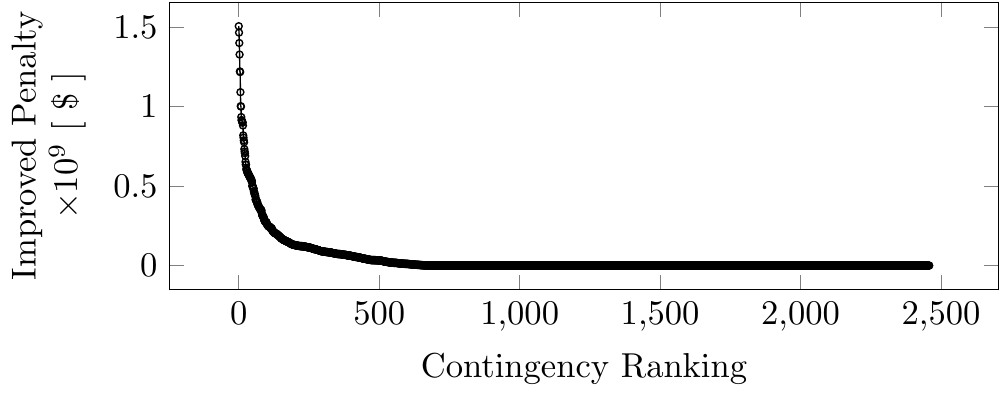}%
}%
\hfill
\subfloat[network $09$]{%
	\includegraphics[width=8cm]{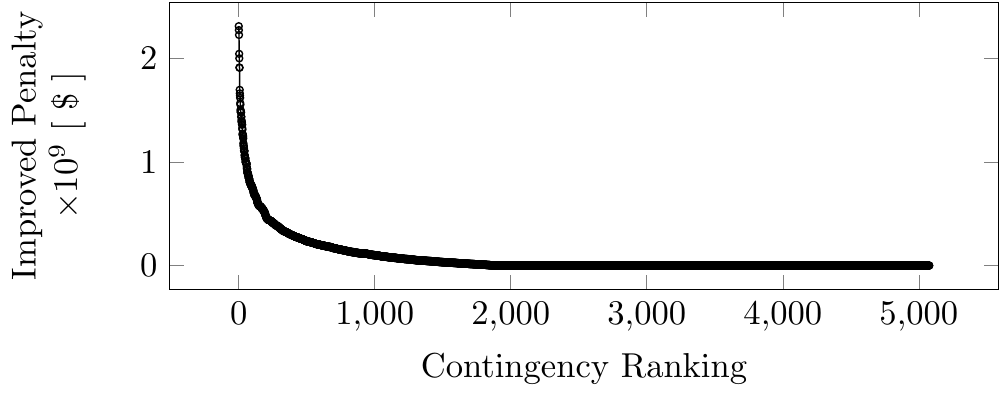}%
}%
\hfill
\subfloat[network $25$]{%
	\includegraphics[width=8cm]{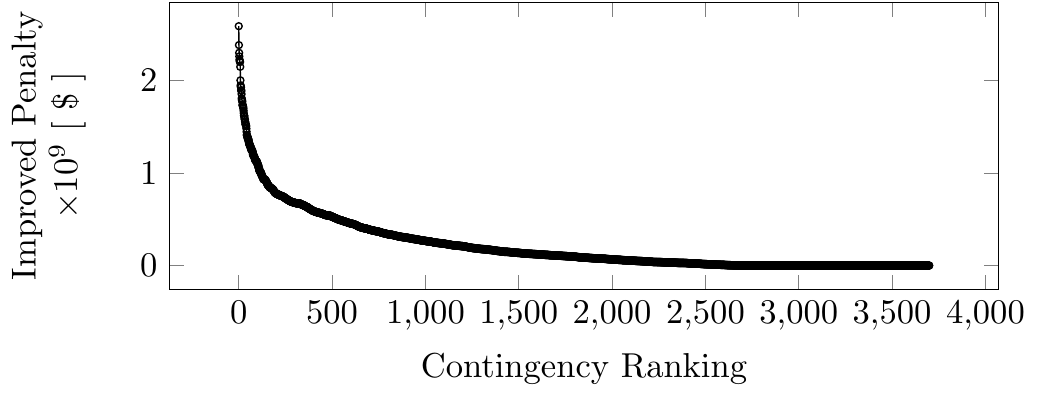}%
}%
\hfill
\caption{Improved penalty of ranked contingencies for different instances of networks $02$, $06$, $09$, and $25$.}
\label{fig: ctg analysis}
\end{figure}

%
%

%% file: 9Appendix.tex
\section{Proofs}

\subsection{Proof of Proposition~\ref{prop:SmoothedRealDisjunction}}\label{sec:Proof:SmoothedRealDisjunction}
\begin{repeattheorem}[Proposition~\ref{prop:SmoothedRealDisjunction}]
    For each \(g\in G_k\) and \(k\in\calK\), as $\epsilon\to 0$, the approximation errors converge to zero, i.e.,
    \begin{align*}
        \sup\left\{\vert{p}_{gk}^\epsilon(\Delta_k)- p_{gk}(\Delta_k)\vert:\Delta_k\in\R\right\}\to 0, \;\; \sup\left\{\vert\tilde{p}_{gk}^{\epsilon}(\Delta_k)- \tilde{p}_{gk}(\Delta_k)\vert:\Delta_k\ge (\underline{p}_g-p_{g0})/\alpha_g\right\}\to 0.
    \end{align*}
\end{repeattheorem}
\begin{proof}{Proof of Proposition~\ref{prop:SmoothedRealDisjunction}.}
    It follows from the definition and the uniform approximation bound \eqref{eq:SmoothedApproximation} that
    \[
        \tilde{p}_{gk}^{\epsilon}(\Delta_k)-\epsilon\ln{2}\le \tilde{p}_{gk}(\Delta_k)\le \tilde{p}_{gk}^{\epsilon}(\Delta_k),\quad\forall\,\Delta_k\in\R.
    \]
    Thus we have \(\sup\{\vert\tilde{p}_{gk}^{\epsilon}(\Delta_k)- p_{gk}^{\epsilon}(\Delta_k)\vert:\Delta_k\ge (\underline{p}_g-p_{g0})/\alpha_g\}\le \epsilon\ln{2}\to 0\) as \(\epsilon\to 0\).
    Moreover, note that \(p_{gk}(\Delta_k)=\underline{p}_g+\max\{0,\tilde{p}_{gk}(\Delta_k)-\underline{p}_g\}\) and \(p^{\epsilon}_{gk}(\Delta_k)=\underline{p}_g+\epsilon\ln(1+\exp((\tilde{p}_{gk}^{\epsilon}(\Delta_k)-\underline{p}_g)/\epsilon))\).
    So the approximation error can be bounded by
    \begin{align*}
        \vert p_{gk}(\Delta_k)-p^{\epsilon}_{gk}(\Delta_k)\vert &\le
        \vert p_{gk}(\Delta_k)-\underline{p}_g-\epsilon\ln(1+\exp((\tilde{p}_{gk}(\Delta_k)-\underline{p}_g)/\epsilon))\vert \\
        &+ \vert \tilde{p}_{gk}(\Delta_k)-\underline{p}_g+\epsilon\ln(1+\exp((\tilde{p}_{gk}(\Delta_k)-\underline{p}_g)/\epsilon))\vert\\
        &\le\epsilon\ln\left(\frac{1+2\exp(\tilde{p}_{gk}(\Delta_k))}{1+\exp(\tilde{p}_{gk}(\Delta_k))}\right)+\epsilon\ln{2}\le 2\epsilon\ln{2}\to 0,\quad\text{as }\epsilon\to 0.
    \end{align*}
    \qed
\end{proof}

\subsection{Proof of Proposition~\ref{prop:SmoothedReactiveDisjunction}}\label{sec:Proof:SmoothedReactiveDisjunction}
\begin{repeattheorem}[Proposition~\ref{prop:SmoothedReactiveDisjunction}]
    For each \(g\in G_k\) and \(k\in\calK\), given any \(\epsilon>0\), we have \(S_{gk}\subseteq {S}_{gk}^{\epsilon}\).
    Moreover, the distance \(\sup_{(q,v)\in {S}_{gk}^{\epsilon}}\inf_{(q',v')\in S_{gk}}\Vert (q,v)-(q',v')\Vert\to 0\) as \(\epsilon\to 0\).
\end{repeattheorem}
\begin{proof}{Proof of Proposition~\ref{prop:SmoothedReactiveDisjunction}.}
    The containment \(S_{gk}\subseteq S_{gk}^{\epsilon}\) follows directly from applying the uniform approximation bound~\eqref{eq:SmoothedApproximation} to constraints~\eqref{eq:SmoothedApproximation:SetConstr2} and~\eqref{eq:SmoothedApproximation:SetConstr3}.
    For notational convenience, let \(x:=(q_g,v_{i_g},v_{i_g}^+,v_{i_g}^-)\) denote the variables,
    \[
    T_{gk}:=\{(q,v,v^+,v^-):v=v_{i_g0}+v^+-v^-,\,\underline{q}_g\le q_{gk}\le \bar{q}_g,\,\underline{v}_{i_g}\le v_{i_g k}\le \bar{v}_{i_g},\,0\le v_{i_g}^+,v_{i_g}^-\le \bar{v}_{i_g}-\underline{v}_{i_g}\}
    \] 
    that is, the compact set defined by the constraints~\eqref{eq:SmoothedApproximation:SetConstr1} and~\eqref{eq:SmoothedApproximation:SetConstr4}, and 
    \begin{align*}
        G_+(x)&:=\min\{q_{gk}-\underline{q}_g,v_{i_g}^+\},\\
        G^\epsilon_+(x)&:=v_{i_g}^+-\epsilon\ln\left[1+\exp\left(\frac{v_{i_g}^+-q_{gk}+\underline{q}_{g}}{\epsilon}\right)\right]-\epsilon\ln2,\\
        G_-(x)&:=\min\{-q_{gk}+\bar{q}_g,v_{i_g}^-\},\\
        G^\epsilon_-(x)&:=v_{i_g}^--\epsilon\ln\left[1+\exp\left(\frac{v_{i_g}^-+q_{gk}-\bar{q}_{g}}{\epsilon}\right)\right]-\epsilon\ln2,
    \end{align*}
    which are the functions corresponding to constraints~\eqref{eq:SmoothedApproximation:SetConstr2}, \eqref{eq:SmoothedApproximation:ApproxConstr2}, \eqref{eq:SmoothedApproximation:SetConstr3}, and \eqref{eq:SmoothedApproximation:ApproxConstr3}, respectively.
    Note that \(S_{gk}\) and \(S_{gk}^{\epsilon}\) are the projections of the compact sets \(T_{gk}\cap\{G_+(x)\le0,\,G_-(x)\le 0\}\) and \(T_{gk}\cap\{G^\epsilon_+(x)\le0,\,G^\epsilon_-(x)\le 0\}\).
    So it suffices to show that any limit point of \(T_{gk}\cap\{G^\epsilon_+(x)\le0,\,G^\epsilon_-(x)\le 0\}\) actually lies in \(T_{gk}\cap\{G_+(x)\le0,\,G_-(x)\le 0\}\) when \(\epsilon\) goes to zero.
    
    Take any sequence \(\{x_{(j)}\}_{j=1}^\infty\) such that \(G_+^{1/j}(x_{(j)})\le0\) and \(G_-^{1/j}(x_{(j)})\le0\) for all \(j\).
    Replacing \(\{x_{(j)}\}_{j=1}^\infty\) with its subsequence if necessary, we may assume that \(x_{(j)}\to x_0\) for some \(x_0\in T_{gk}\) due to the compactness.
    We claim that \(G_+(x_0)\le0\) and \(G_-(x_0)\le 0\) hold.
    Otherwise suppose \(\delta:=\max\{G_+(x_0),G_-(x_0)\}>0\).
    We can find \(j>0\) such that \(\Vert x_{(j)}-x_0\Vert< \delta/4\), \(\sup_{x\in T_{gk}}\vert G_+(x)-G_+^{1/j}(x)\vert< \delta/2\), and \(\sup_{x\in T_{gk}}\vert G_-(x)-G_-^{1/j}(x)\vert< \delta/2\).
    Thus we have \(G_+(x_0)\le\vert G_+(x_0)-G_+(x_{(j)})\vert + \vert G_+(x_{(j)})-G_+^{1/j}(x_{(j)})\vert< 2\cdot(\delta/4)+\delta/2=\delta\) and the same holds for \(G_-(x_0)\), which leads to a contradiction.
    Therefore, we see that \(x_0\in T_{gk}\cap\{G_+(x)\le0,\,G_-(x)\le 0\}\) and this completes the proof. \qed
\end{proof}

\subsection{Proof of Theorem \ref{thm:admm_iter_complex}}\label{sec:proof_of_admm}
\begin{repeattheorem}[Theorem \ref{thm:admm_iter_complex}]
	Let $\lmd=\{\lmd_k\}_{k\in \cal{K}}$ and $\beta>0$ be given, and $\epsilon >0$. Suppose ADMM generates  iterates $\{\x^t_0,\{\x^t_k, (\x^{\text{base}}_0)^{t}\}_{k\in \cal{K}}, \{\z^t_k\}_{k\in \cal{K}} \}_{t\in \N}$ that satisfy the following conditions:
	\begin{itemize}
		\item[-] (descent in base case update) $\x_0^{t+1}$ is a stationary point of the optimization problem in \eqref{eq:distributed_admm_first} satisfying 
		\begin{align}
			& f_0(\x^{t+1}_0) + \sum_{k\in \mathcal{K}} \left(\langle \y_k^t, \x^{t+1}_0 \rangle + \frac{\rho}{2}\| \x^{t+1}_0 - (\x^{\text{base}}_k)^t + {\z}^t_k\|^2\right) \notag \\
			\leq &  f_0(\x^{t}_0) + \sum_{k\in \mathcal{K}} \left(\langle \y_k^t, \x^{t}_0 \rangle + \frac{\rho}{2}\| \x^{t}_0 - (\x^{\text{base}}_k)^t + {\z}^t_k\|^2\right);\label{eq:descent_assumption_base}
		\end{align} 
		\item[-] (descent in contingenices update) there exists a positive constant $\gamma >0$ such that for all $k \in \mathcal{K}$, $\x^{t+1}_k$ is a stationary point of the optimization problem in \eqref{eq:distributed_admm_second} satisfying
		\begin{align}
			& f_k(\x^{t+1}_k) - \langle \y_k^t, (\x^{\text{base}}_k)^{t+1} \rangle + \frac{\rho}{2}\| \x^{t+1}_0 - (\x^{\text{base}}_k)^{t+1} + {\z}^t_k\|^2 \notag \\
			\leq & f_k(\x^{t}_k) - \langle \y_k^t, (\x^{\text{base}}_k)^{t} \rangle + \frac{\rho}{2}\| \x^{t+1}_0 - (\x^{\text{base}}_k)^{t} + {\z}^t_k\|^2 - \gamma\beta \|(\x^{\text{base}}_k)^{t+1}-(\x^{\text{base}}_k)^{t}\|^2. \label{eq:descent_assumption_ctg}
		\end{align}
	\end{itemize}
	Then ADMM finds an $\epsilon$-stationary solution of the ALR problem \eqref{eq:three_block_relaxation} in the sense of definition \ref{def:alr_approx_solution} in at most
	\begin{align}\label{eq:admm_iter_complex}
		T\leq \left\lceil  \frac{ 2\rho^2 |\cal{K}| (\overline{L}(\lmd,\beta)-\underline{L}(\lmd,\beta))}{\epsilon^2 \min \{\gamma\beta, (\beta+\rho)/2-\beta^2/\rho \} } \right\rceil	
	\end{align}
	iterations, where
	\begin{align}
		\overline{L}(\lmd,\beta) := & f_0(\x^0_0) + \sum_{k \in \mathcal{K}} \left(f_k(\x^0_k)	+\langle \lmd_k, {\z}^0_k \rangle + \frac{\beta}{2}\|{\z}^0_k\|^2 \right) \notag \\
	& + \sum_{k \in \mathcal{K}} \left(\langle \y^0_k, \x^0_0 - (\x^{\text{base}}_k)^0+ {\z}^0_k \rangle + \frac{\rho}{2}\|\x^0_0 - (\x^{\text{base}}_k)^0+ {\z}^0_k\|^2 \right), \\
		\underline{L}(\lmd,\beta):= & \min \left\{\sum_{g\in \cal{G}} c_g(p_{g0}):\ p_{g0}\in [\underline{p}_g, \overline{p}_g], \ \forall g\in \cal{G} \right\} - \frac{\|\lmd\|^2}{\beta}.
	\end{align}
\end{repeattheorem}
\proof{Proof of Theorem \ref{thm:admm_iter_complex}.} We prove the theorem by steps.
	\begin{enumerate}
	\item{Investigate optimality conditions of subproblem updates.} \\
	The optimality conditions of updates in \eqref{eq:distributed_admm_first}-\eqref{eq:distributed_admm_third} are given as 
	\begin{subequations}
	\begin{align}
	\d^{t+1}_0:= &  \rho \sum_{k\in \cal{K}} \left( ({\x}^{\text{base}}_k)^{t} - ({\x}^{\text{base}}_k)^{t+1}\right) + \left(\z^{t+1}_k-\z^t_k\right) \in  \partial f({\x}^{t+1}_0) + \sum_{k\in \cal{K}} \y^{t+1}_{k} + \mathcal{N}_{X_0}({\x}^{t+1}_0),\\
		\d^{t+1}_{k} := & \begin{bmatrix} 0\\ \rho\left(\z^{t+1}_k-\z^t_k\right) \end{bmatrix} \in  \begin{bmatrix} \tilde{\nabla} f_k({\x}_k^{t+1})\\- \y_k^{t+1}\end{bmatrix}	+ \mathcal{N}_{X_k}\left({\x}^{t+1}_k, ({\x}^{\text{base}}_k)^{t+1}\right), \ \forall k\in \cal{K},\label{eq:ctg_dual_res}\\
		0 = & \lmd_k + \beta \z^{t+1}_k + \y_k^{t} + \rho({\x}^{t+1}_0- ({\x}^{\text{base}}_k)^{t} + \z_k^{t+1}) = \lmd + \beta \z^{t+1}_k + \y_k^{t+1} , \ \forall k\in \cal{K},\label{eq:z_opt_condition}
	\end{align}
	\end{subequations}
	where $\tilde{\nabla} f_k({\x}_k^{t+1})\in \partial f_k({\x}_k^{t+1})$ in \eqref{eq:ctg_dual_res}. The dual update \eqref{eq:centralized_admm_dual} and \eqref{eq:z_opt_condition} further imply that
	\begin{align}
		\s^{t+1}_k := {\x}^{t+1}_0- ({\x}^{\text{base}}_k)^{t} + \z_k^{t+1} = \frac{1}{\rho}(\y^{t+1}_k-\y^{t}_k)=\frac{\beta}{\rho}(\z^{t}_k-\z^{t+1}_k)\ \forall k\in \cal{K}.
	\end{align}
	Recall $\rho >\beta > 1$, and hence $\beta/\rho  < 1 < \rho$. As a result, 
	\begin{align}\label{eq:bound_on_res}
		 & \max \left\{ \|\d^{t+1}_0\|, \|\d^{t+1}_1\|, \cdots, \|\d^{t+1}_{|\mathcal{K}|}\|,\|\s^{t+1}_1\|, \cdots, \|\s^{t+1}_{|\cal{K}|}\| \right \} \notag \\
		\leq & \rho \sum_{k\in \cal{K}} \left(\|({\x}^{\text{base}}_k)^{t+1} - ({\x}^{\text{base}}_k)^{t}\|+  \|\z^{t+1}_k-\z^t_k\|\right).
	\end{align}

	\item{Lower-boundedness of the augmented Lagrangian function.}\\
	Recall the augmented Lagrangian function defined in \eqref{eq:al_function}. For some fixed $\lmd$ and $\beta$, denote $$L^t(\lmd, \beta) := L_\rho(\x^t_0, \{\x^t_k, (\x^{\text{base}}_k)^{t} \}_{k\in \mathcal{K}}, \{{\z}^t_k\}_{k\in \mathcal{K}}, \{\y^t_k\}_{k\in \mathcal{K}})$$
	for $t\in \N$. We first show that $L^t(\lmd, \beta)$ is bounded from below. Notice that if we let $h_k(\z_k) := \langle \lmd_k, \z_k\rangle +\frac{\beta}{2}\|\z_k\|^2$, then it holds that $h_k(\z_k) + \nabla h_k(\z_k)^\top (\u - \z_k) \geq h_k(\u) - \frac{\beta}{2}\|\u-\z_k\|^2$ for any $u$, since $h_k$ has $\beta$-Lipchitz gradient. Consequently, for all $k\in \cal{K}$ and $t\in \N$, 
	\begin{align}\label{eq:smoothness_of_norm_squre}
	 & \langle \lmd_k, \z_k^t \rangle + \frac{\beta}{2}\|\z_k^t\|^2 + \langle \y^t_k, \x_0^t-(\x^{\text{base}}_{k})^{t}+{\z}_k^t \rangle \notag\\
		\geq & \langle \lmd_k, (\x^{\text{base}}_{k})^{t}-\x_0^t \rangle + \frac{\beta}{2}\|\x_0^t-(\x^{\text{base}}_{k})^{t}\|^2 - \frac{\beta}{2}\|\x_0^t-(\x^{\text{base}}_{k})^{t}+{\z}_k^t\|^2,
	\end{align}
	where we use the fact that $\y_k^t = -\lmd_k-  \beta {\z}_k^t =-\nabla h_k({\z}_k^t)$, and $\u = (\x^{\text{base}}_{k})^{t}-\x_0^t$. Using \eqref{eq:smoothness_of_norm_squre} inside the definition of $L^t(\lmd, \beta)$, we have
	\begin{align}\label{eq:alr_lower_bounded}
		L^t(\lmd, \beta) \geq & f_0(\x_0^t) +\sum_{k\in \cal{K}}\left( f_k(\x^t_k)+ \langle \lmd_k, (\x^{\text{base}}_{k})^{t}-\x_0^t \rangle + \frac{\beta}{2}\|\x_0^t- (\x^{\text{base}}_{k})^{t}\|^2 + \frac{\rho-\beta}{2}\|\x_0^t- (\x^{\text{base}}_{k})^{t}+{\z}_k^t\|^2\right) \notag \\
		\geq &  f_0(\x_0^t) +\sum_{k\in \cal{K}} \left(f_k(\x^t_k)- \frac{\|\lmd_k\|^2}{2\beta}\right) \notag \\
		\geq & \min \left\{\sum_{g\in \cal{G}} c_g(p_{g0}):\ p_{g0}\in [\underline{p}_g, \overline{p}_g], \ \forall g\in \cal{G} \right\} - \frac{\|\lmd\|^2}{\beta} := \underline{L}(\lmd, \beta),
	\end{align}
	where the second inequality is due to $$\langle \lmd_k, (\x^{\text{base}}_{k})^{t}-\x_0^t \rangle + \frac{\beta}{2}\|\x_0^t- (\x^{\text{base}}_{k})^{t}\|^2 = \frac{\beta}{2}\|\x_0^t- (\x^{\text{base}}_{k})^{t} -\frac{\lmd_k}{\beta}\|^2 - \frac{\|\lmd_k\|^2}{2\beta},$$
	and the last inequality is due to the definitions of $f_0$ and $f_k$ in \eqref{eq:abstrac_obj}.
	
	\item{Monotonicity of the augmented Lagrangian function.}\\
	Next we show the sequence $\{L^t(\lmd,\beta)\}_{t\in \Z}$ is non-increasing. By the dual update \eqref{eq:centralized_admm_dual}, we have
	\begin{align}
		L^{t+1}(\lmd,\beta) = & L_\rho(\x^{t+1}_0, \{\x^{t+1}_k, (\x^{\text{base}}_k\}_{k\in \mathcal{K}})^{t+1}, \{{\z}^{t+1}_k\}_{k\in \mathcal{K}}, \{\y^t_k\}_{k\in \mathcal{K}}) + \frac{1}{\rho} \sum_{k\in \cal{K}} \|\y^{t+1}_k-\y_k^t\|^2 \notag \\
		= &  L_\rho(\x^{t+1}_0, \{\x^{t+1}_k, (\x^{\text{base}}_k\}_{k\in \mathcal{K}})^{t+1}, \{{\z}^{t+1}_k\}_{k\in \mathcal{K}}, \{\y^t_k\}_{k\in \mathcal{K}}) + \frac{\beta^2}{\rho} \sum_{k\in \cal{K}} \|\z^{t+1}_k-{\z}_k^t\|^2  \notag \\
		= & \min_{\{{\z}_k\}_{k\in\cal{K}}} L_\rho(\x^{t+1}_0, \{\x^{t+1}_k, (\x^{\text{base}}_k\}_{k\in \mathcal{K}})^{t+1}, \{{\z}_k\}_{k\in \mathcal{K}}, \{\y^t_k\}_{k\in \mathcal{K}}) + \frac{\beta^2}{\rho} \sum_{k\in \cal{K}} \|\z^{t+1}_k-{\z}_k^t\|^2 \notag \\
		\leq &  L_\rho(\x^{t+1}_0, \{\x^{t+1}_k, (\x^{\text{base}}_k\}_{k\in \mathcal{K}})^{t+1}, \{{\z}^{t}_k\}_{k\in \mathcal{K}}, \{\y^t_k\}_{k\in \mathcal{K}}) + \frac{\beta^2}{\rho} \sum_{k\in \cal{K}} \|\z^{t+1}_k-{\z}_k^t\|^2 - \frac{\beta+\rho}{2}\|{\z}_k^{t+1}-{\z}_k^t\|^2,\label{eq:descent_in_z}
	\end{align}
	where the second equality is due to \eqref{eq:z_opt_condition}, and the last inequality is due $\{\z^{t+1}_k\}_{k\in \cal{K}}$ being the minimizer of a $(\beta+\rho)$-strongly convex function. Assumptions \eqref{eq:descent_assumption_base}-\eqref{eq:descent_assumption_ctg} and \eqref{eq:descent_in_z} together imply that
	\begin{align}\label{eq:alr_descent_sequence}
		 \sum_{k\in \cal{K}} \gamma \beta \|(\x_k^{\text{base}})^{t+1} - (\x_k^{\text{base}})^{t}\|^2 + \left(\frac{\beta+\rho}{2} - \frac{\beta^2}{\rho}\right)\|{\z}_k^{t+1} -{\z}_k^t\|^2 \leq L^t(\lmd, \beta)- L^{t+1}(\lmd, \beta). 
	\end{align}
	
	\item{Iteration Complexity.}\\
	For any positive integer $T$, there exists some index $0\leq t\leq T-1$ such that
	\begin{align}
		& \max \left\{ \|\d^{t+1}_0\|, \|\d^{t+1}_1\|, \cdots, \|\d^{t+1}_{|\mathcal{K}|}\|,\|\s^{t+1}_1\|, \cdots, \|\s^{t+1}_{|\cal{K}|}\| \right \} \notag \\
		 \leq & \rho \sum_{k\in \cal{K}} \left(\|(\x_k^{\text{base}})^{t+1} - (\x_k^{\text{base}})^{t}\| + \|{\z}_k^{t+1} -{\z}_k^t\| \right)\notag \\
		\leq & \rho (2|\cal{K}|)^{1/2}\left( \sum_{k\in\cal{K}} \|(\x_k^{\text{base}})^{t+1} - (\x_k^{\text{base}})^{t}\|^2 + \|{\z}_k^{t+1} -{\z}_k^t\|^2 \right)^{1/2} \notag \\
		\leq & \rho (2|\cal{K}|)^{1/2} \left(\frac{\overline{L}(\lmd,\beta) - \underline{L}(\lmd,\beta)}{T \min\{\gamma\beta, (\beta+\rho)/2-\beta^2/\rho\}}\right)^{1/2}.\label{eq:final_bound_on_res}
	\end{align}
	where the first inequality is due to \eqref{eq:bound_on_res}, the second inequality is due to $\sum_{i=1}^n a_i=e^\top a \leq \|e\|\|a\|$ for $a\in \R^n$ and $e=[1,\cdots, 1]^\top\in \R^n$, and the last inequality is due to \eqref{eq:alr_lower_bounded} and \eqref{eq:alr_descent_sequence}. Finally, plugging the upper bound \eqref{eq:admm_iter_complex} of $T$ into \eqref{eq:final_bound_on_res} proves the claim. \qed 
	\end{enumerate}
\endproof

\subsection{Proof of Theorem \ref{thm:alm_iter_complex}}\label{sec:proof_of_alm}
\begin{repeattheorem}[Theorem \ref{thm:alm_iter_complex}]
	Let $\epsilon >0$. Suppose the assumptions in Theorem \ref{thm:admm_iter_complex} hold, and each ADMM returns an $\epsilon$-stationary solution of the ALR problem (see Definition \ref{def:alr_approx_solution}). 
	Moreover, assume there exists some $0 < \overline{L} < +\infty$ such that $\overline{L}\geq \overline{L}(\lmd^r, \beta_r)$ for all $r\in \N$. Define constants
	\begin{itemize}
		\item[-] $\eta := \min\{\gamma, \frac{\tau+1}{2} - \frac{1}{\tau}\}$,
		\item[-] $\Lambda := \max \left\{ \sum_{k\in \cal{K}} \|\lmd_k\|^2 : \ \lmd_k \in [\underline{\lmd}_k, \overline{\lmd}_k], \ \forall k\in \cal{K}\right\}$,
		\item[-] $\underline{L} :=  \min \left\{\sum_{g\in \cal{G}} c_g(p_{g0}):\ p_{g0}\in [\underline{p}_g, \overline{p}_g], \ \forall g\in \cal{G} \right\}- \frac{\Lambda}{\beta_0}$, and $\delta_L :=\overline{L}- \underline{L}$.
	\end{itemize}
	Then the two-level ADMM (Algorithm \ref{alg:two_level}) finds an $\epsilon$-stationary solution of problem \eqref{eq:two_block_scopf} in at most $R$ outer-level updates, where
	\begin{align}\label{eq:outer_level_bd}
		R \leq \left \lceil   \log_c \left( \frac{4 \delta_L}{\beta_0  \epsilon^2} \right)\right \rceil,
	\end{align}
	 and a total of 
	\begin{align}\label{eq:total_iter_complex}
		\left\lceil \left( \frac{2c\tau^2 \beta_0|\cal{K}|\delta_L}{(c-1) \eta}\right)\left( \frac{c^R-1}{\epsilon^2} \right)\right \rceil + R	=\mathcal{O}(\epsilon^{-4})
	\end{align}
	inner-level ADMM iterations.
\end{repeattheorem}
\proof{Proof of Theorem \ref{thm:alm_iter_complex}.} We use $T_r$ to bound the number of iterations required by the $r$-th ADMM to return an $\epsilon$-stationary solution in the sense of Definition \ref{def:alr_approx_solution}. By Theorem \ref{thm:admm_iter_complex}, and  the fact that $\rho = \tau \beta_r = \tau \beta_0 c^r$, we know $T_r$ can be bounded by
\begin{align}
	 T_r \leq 	\frac{2\tau^2 \beta_0 |\mathcal{K}| \delta_L}{\epsilon^2 \eta } c^r + 1. \label{eq:single_admm_iter}
\end{align}
Let $R$ denote number of outer-level updates. Then by \eqref{eq:single_admm_iter}, the total number of inner ADMM iterations is bounded by 
\begin{align*}
	\sum_{r=1}^R T_k \leq \left\lceil \left( \frac{2c\tau^2 \beta_0|\cal{K}|\delta_L}{(c-1) \eta}\right)\left( \frac{c^R-1}{\epsilon^2} \right)\right \rceil + R.
\end{align*}
This proves the left-hand side of \eqref{eq:total_iter_complex}.

Next we bound the number of outer-level updates $R$. For $r\in [R]$, we use $\d_0^r$, $\d_1^r, \cdots, \d^r_{|\cal{K}|}$, and $\s_1^r,\cdots, \s^r_{|\cal{K}|}$ to denote the dual and primal residuals of ADMM corresponding to the solution $(\x^{r}_0, \{ \x^{r}_k, (\x^{\text{base}}_k)^{r}\}_{k\in \cal{K}},\{{\z}^{r}_k\}_{k\in \cal{K}})$. Notice that at the end of the $R$-th ADMM, we already have $\max\{\|\d_0^R\|, \|\d^R_1\|,\cdots, \|\d^R_{|\cal{K}|}\|\} \leq \epsilon$, so it remains to find a large enough index $R$ such that the quantity $\max\{\|\r^R_1\|,\cdots, \|\r^R_{|\cal{K}|}\|\} \leq \epsilon$ as well. By the first inequality in \eqref{eq:alr_lower_bounded} and the fact that $\langle \lmd_k, \r_k^R \rangle \geq -\frac{\|\lmd_k\|^2}{\beta_R} - \frac{\beta_R}{4} \|\r_k^R\|^2$, we have
\begin{align*}
	\overline{L} \geq &  \min \left\{\sum_{g\in \cal{G}} c_g(p_{g0}):\ p_{g0}\in [\underline{p}_g, \overline{p}_g], \ \forall g\in \cal{G} \right\} + \sum_{k\in \cal{K}}  -\frac{\|\lmd_k\|^2}{\beta_R} - \frac{\beta_R}{4} \|\r_k^R\|^2 + \frac{\beta_R}{2}\|r^R_k\|^2 \\
	\geq & \underline{L} + \frac{\beta_0 c^R} {4}\sum_{k\in \cal{K}} \|\r_k^R\|^2,
\end{align*}
which further implies that 
\begin{align*}
	\|\r_k^R\| \leq \left( \frac{4\delta_L}{\beta_0 c^R}\right)^{1/2}, \ \forall k\in \cal{K}.
\end{align*}
The claimed upper bound of $R$ in \eqref{eq:outer_level_bd} ensures $\|\r_k^R\|\leq \epsilon$ for all $k\in \cal{K}$. Notice that $c^R = \mathcal{O}(\epsilon^{-2})$, plugging which into the left-hand side of \eqref{eq:total_iter_complex} gives the $\mathcal{O}(\epsilon^{-4})$ upper bound on the total number of ADMM iterations. This completes the proof. \qed 
\endproof

\section{Dataset Characteristics and Further Data}
\subsection{Test Systems}
Table \ref{tab: network stat} provides the characteristics of $17$ network models adopted for computational experiments in the paper. Each row of this table corresponds to a network with $20$ different instances. The number of shunts in the table is the sum of both fixed and switched shunts. See \cite{GODataSet} for more details.
\begin{table}[]
\centering
\caption{Dataset characteristics.}
\label{tab: network stat}
\begin{tabular}{|c@{\hskip 0.02in}|c@{\hskip 0.02in}|c@{\hskip 0.02in}|c@{\hskip 0.02in}|c@{\hskip 0.02in}|c@{\hskip 0.02in}|c@{\hskip 0.02in}|c@{\hskip 0.01in}|}
\hline
\textbf{Network} & $|\calI|$ & \textbf{Loads} & \textbf{Shunts} & $|\calG|$         & $|\calE|$      & $|\calF|$     & $|\calK|$ \\ \hline \hline
Net. 02     & 500            & 281            & 36              & 224               & 540            & 193           & 727$\pm$34                 \\ \hline
Net. 03     & 793            & 568            & 99              & 214               & 770            & 143           & 92$\pm$8                   \\ \hline
Net. 06     & 2000           & 1010           & 124             & 384               & 2743           & 896           & 2478$\pm$57                \\ \hline
Net. 70     & 2312           & 1529           & 322             & 444               & 2156           & 857           & 1009$\pm$28                \\ \hline
Net. 75     & 2742           & 1832$\pm$2     & 15              & 186$\pm$6         & 3070$\pm$1     & 1602          & 2157$\pm$6                 \\ \hline
Net. 08     & 3022           & 1793           & 531             & 637               & 2838           & 1297          & 1924$\pm$48                \\ \hline
Net. 82     & 4285$\pm$316   & 3185$\pm$317   & 15$\pm$2        & 395$\pm$13        & 4825$\pm$320   & 2096$\pm$42   & 2509$\pm$19                \\ \hline
Net. 83     & 4020           & 2691           & 43              & 348$\pm$4         & 4636$\pm$14    & 2338          & 3018$\pm$15                \\ \hline
Net. 86     & 4619           & 3291$\pm$1     & 205$\pm$1       & 365$\pm$19        & 5739           & 2412$\pm$1    & 3043$\pm$18                \\ \hline
Net. 88     & 4837$\pm$1     & 2753           & 48$\pm$1        & 329$\pm$3         & 4746           & 3017$\pm$2    & 3432$\pm$3                 \\ \hline
Net. 09     & 4918           & 3070           & 729             & 1340              & 4412           & 2315          & 5076$\pm$4                 \\ \hline
Net. 12     & 9154$\pm$437   & 6239$\pm$420   & 114$\pm$7       & 365               & 10427$\pm$500  & 4942$\pm$46   & 5152$\pm$151               \\ \hline
Net. 13     & 10000          & 3984           & 564             & 2089              & 10819          & 2374          & 6071$\pm$3605              \\ \hline
Net. 14     & 10480          & 6860$\pm$11    & 248             & 777$\pm$5         & 12741$\pm$2    & 5186          & 8559$\pm$7                 \\ \hline
Net. 20     & 19139$\pm$263  & 12686$\pm$242  & 2442$\pm$10     & 947$\pm$26        & 22474$\pm$476  & 11585$\pm$169 & 13274$\pm$118              \\ \hline
Net. 25     & 24464$\pm$1    & 15810$\pm$1    & 19341$\pm$28    & 1589$\pm$2        & 27467$\pm$5    & 10345$\pm$1   & 3701$\pm$2                 \\ \hline
Net. 30     & 30000          & 10648          & 1247$\pm$28     & 3526              & 32020          & 3373          & 16070$\pm$5890             \\ \hline
\end{tabular}
\end{table}
\subsection{Penalty Function and Model Parameters}
\label{sec:Weights_and_Penalty Function_Parameters}
As mentioned in Section \ref{sec:SC-ACOPF_Model}, $c_{ik}^p(\cdot)$, $c_{ik}^q(\cdot)$, and $\sigma_{ek}^S(\cdot)$ are convex pwl increasing functions with three pieces. Let $s_{base}$ denote the base power which varies for different networks, and its value is provided in the dataset. For all of the above functions,
we set the length of the three pieces to $2/s_{base}$, $50/s_{base}$, and $\infty$, respectively. Moreover, the linear coefficient of the pieces are set to $s_{base}\times10^3$, $5s_{base}\times10^3$, and $s_{base}\times10^6$, respectively.
The approximation parameter in~\eqref{eq:SmoothedRealPowerDisjunction}, \eqref{eq:real_power_smooth_approx_lower_bounded}, and \eqref{eq:SmoothedConstraintApproximation} is set to \(\epsilon=10^{-6}\) to heuristically balance the trade-off between approximation accuracy and numerical condition.

\subsection{Contingency Selection and Algorithmic Parameters in the Two-level ADMM}\label{sec:admm_parameters}
In our experiments presented in Section \ref{subsec: admm}, we replace constraint \eqref{eq:scopf_realpowercntg} by the big-M relaxation of an equivalent mixed-integer representation (e.g., see \cite{GOChallenge1_formulation}) in the definition of $X_k$ for all $k\in \cal{K}'$, as we observed that the smooth approximation introduced in Section \ref{sec:smoothing} is more prone to numerical failures during ADMM iterations for the tested SC-ACOPF instances.
For the two-level algorithm, we choose $\tau = 2$, $\beta_0=2000$. The inner level is terminated if $\max_{k\in \cal{K}}\{\|\x^{t+1}_0 - (\x^{\text{base}}_k)^{t+1} + {\z}_k^{t+1}\|_{\infty}\}\leq 0.1/r$, where $r$ is the current outer iteration index. Such an inexact termination encourages more frequent updates on outer-level dual variables and penalty, which are updated according to the last paragraph in Section \ref{sec:two_level_admm}: we set $\beta^{r+1}=8\beta^r$ if $\max_{k\in \cal{K}'}\{\|\x^{r}_0 - (\x^{\text{base}}_k)^{r}\|_{\infty}\}> 0.5 \max_{k\in \cal{K}'}\{\|\x^{r-1}_0 - (\x^{\text{base}}_k)^{r-1}\|_{\infty}\}$, and $\beta^{r+1}=\beta^r$ otherwise.